\documentclass{amsart}
\usepackage{amssymb}
\usepackage{graphicx}
\usepackage{url}
 \newtheorem{theorem}{Theorem}
 \newtheorem{lemma}[theorem]{Lemma}
 \newtheorem{proposition}[theorem]{Proposition}
 \newtheorem{corollary}[theorem]{Corollary}
 \newtheorem{remark}[theorem]{Remark}
 \newtheorem{example}[theorem]{Example}
 \newtheorem{definition}[theorem]{Definition}
 \newtheorem{conjecture}[theorem]{Conjecture}

 \newtheorem{question}[theorem]{Question}

\newcommand{\bpr}{\begin{proof}}
\newcommand{\epr}{\end{proof}}
\newcommand{\beq}{\begin{equation}}
\newcommand{\eeq}{\end{equation}}
\newcommand{\bThm}{\begin{theorem}}
\newcommand{\eThm}{\end{theorem}}
\newcommand{\blem}{\begin{lemma}}
\newcommand{\elem}{\end{lemma}}
\newcommand{\bpro}{\begin{proposition}}
\newcommand{\epro}{\end{proposition}}
\newcommand{\bcor}{\begin{corollary}}
\newcommand{\ecor}{\end{corollary}}
\newcommand{\brem}{\begin{remark}}
\newcommand{\erem}{\end{remark}}
\newcommand{\bexa}{\begin{example}}
\newcommand{\eexa}{\end{example}}
\newcommand{\bdf}{\begin{definition}}
\newcommand{\edf}{\end{definition}}
\newcommand{\bcon}{\begin{conjecture}}
\newcommand{\econ}{\end{conjecture}}
\newcommand{\bque}{\begin{question}}
\newcommand{\eque}{\end{question}}

 \newcommand{\Z}{{\mathbb Z}}

 \newcommand{\R}{{\mathbb R}}

\newcommand{\p}{\partial}

\newcommand{\comment}[1]{}

\title{Virtual, Welded, and Ribbon Links in Arbitrary Dimensions}
\author{Blake K. Winter}
\address{}

\date{}

\begin{document}
\thispagestyle{empty}

\begin{abstract}
We define a generalization of virtual links to arbitrary dimensions by extending the geometric definition due to Carter et al. We show that many homotopy type invariants for classical links extend to invariants of virtual links. We also define generalizations of virtual link diagrams and Gauss codes to represent virtual links, and use such diagrams to construct a combinatorial biquandle invariant for virtual $2$-links. In the case of $2$-links, we also explore generalizations of Fox-Milnor movies to the virtual case. In addition, we discuss definitions extending the notion of welded links to higher dimensions. For ribbon knots in dimension $4$ or greater, we show that the knot quandle is a complete classifying invariant up to taking connected sums with a trivially knotted $S^1\times S^{n-1}$, and that all the isotopies involved may be taken to be generated by stable equivalences of ribbon knots.
\end{abstract}

\maketitle

%
\section{Introduction}
%

Let us begin by defining the following terms, which will be used throughout. All manifolds are taken to be smooth, unless otherwise specified. An immersed embedding, $\imath: L\hookrightarrow M$, of an $n$-manifold $L$ into an $(n+2)$-manifold $M$ will be called a \emph{$n$-link in M}. We will refer to this as simply a \emph{link} when the dimension and ambient space are implied by context. When $M=S^{n+2}$, it may also be termed a \emph{classical link}. A link for which $L$ has only one connected component will be called a \emph{$n$-knot (in $M$)}. Note that in some sources, the term $n$-knot is applied only to embeddings of the form $S^{n}\rightarrow M^{n+2}$; we use the term in the more general sense. Two links are considered equivalent if they are related by a smooth isotopy of their images. This is equivalent to requiring that their images be related by an ambient isotopy of $M$, \cite{Roseman2}. Note that this definition of links is also called the \emph{smooth category} of links. If the requirement of immersion is replaced by requiring the embedding to be piecewise-linear and locally flat, the resulting theory is called the \emph{piecewise-linear} or \emph{PL category}; if we instead require the embeddings to be merely continuous and locally flat, we obtain the \emph{topological category} or \emph{TOP}. We will always work with the smooth category unless otherwise specified.

The concept of virtual links was introduced by Kauffman, \cite{Kauff}, as a generalization of classical links. These virtual links were given geometric interpretations by Kamada, Kamada, and Carter et al., \cite{KamaKama, StableEq} and a somewhat different geometric interpretation by Kuperberg, \cite{Kuper}. By using the geometric interpretation given by Kamada et al. as our starting point, we will give a generalization of the definition which applies to codimension-$2$ embeddings in any dimension. Using Roseman's work on projections for links, \cite{Roseman2}, we define higher-dimensional generalizations of virtual link diagrams and Gauss codes. We will also discuss some methods for generalizing welded links to higher dimensions, using the work of Satoh, \cite{SS}, as our starting point.

\section{Link Diagrams}



Every classical $1$-link can be expressed by a \emph{link diagram}, which is a combinatorial object consisting of a planar graph on a surface whose vertices are $4$-valent. At each vertex, two of the edges of the graph are marked as \emph{the overcrossing arc} and the other two are marked as the \emph{undercrossing}. These markings give information on how to resolve them in three dimensions.

Let  $\imath: L\hookrightarrow F\times [0,1]$ be a $1$-link, where $F$ is a surface (possibly with boundary). Then there are canonical projections $\pi:F\times [0,1]\rightarrow F$ and $\pi_I:F\times [0,1]\rightarrow [0,1]$. It is possible to isotope $\imath: L\hookrightarrow F\times I$ such that $\pi \imath$ is an immersion and an embedding except for a finite number of double points called \emph{crossings}. 
In order to recover $L$, at each double point, we mark which strand projects under $\pi_I$ to the smaller coordinate in $[0,1]$ at that double point; this is the undercrossing. These marks allow us to recover $L$ up to ambient isotopy. Observe that $\pi(L)$ is a $4$-valent graph. The interiors of the edges of this graph will be called \emph{semi-arcs} of the link diagram. Let $D_{-}$ be the set $\{x\in L | \pi^{-1}(\pi(x))=\{x, y\}, x \neq y, \pi_I(x)<\pi_I(y)\}$. For any connected component $A$ of $L-D_-$, $\pi(A)$ is called an \emph{arc} of the link diagram.

For diagrams of $1$-links, there exists a set of three moves, the \emph{Reidemeister moves}, which satisfy the following condition: if $L$ and $L'$ are ambient isotopic links in $F\times [0,1]$, then the diagrams of $L$ and $L'$ on $F$ differ by a finite sequence of ambient isotopies of $F$ and Reidemeister moves. Thus, questions about $1$-links may be translated into questions about equivalence classes of link diagrams, with the equivalence relation being generated by the Reidemeister moves.

\begin{figure}
		\centering
			\includegraphics[scale=0.3]{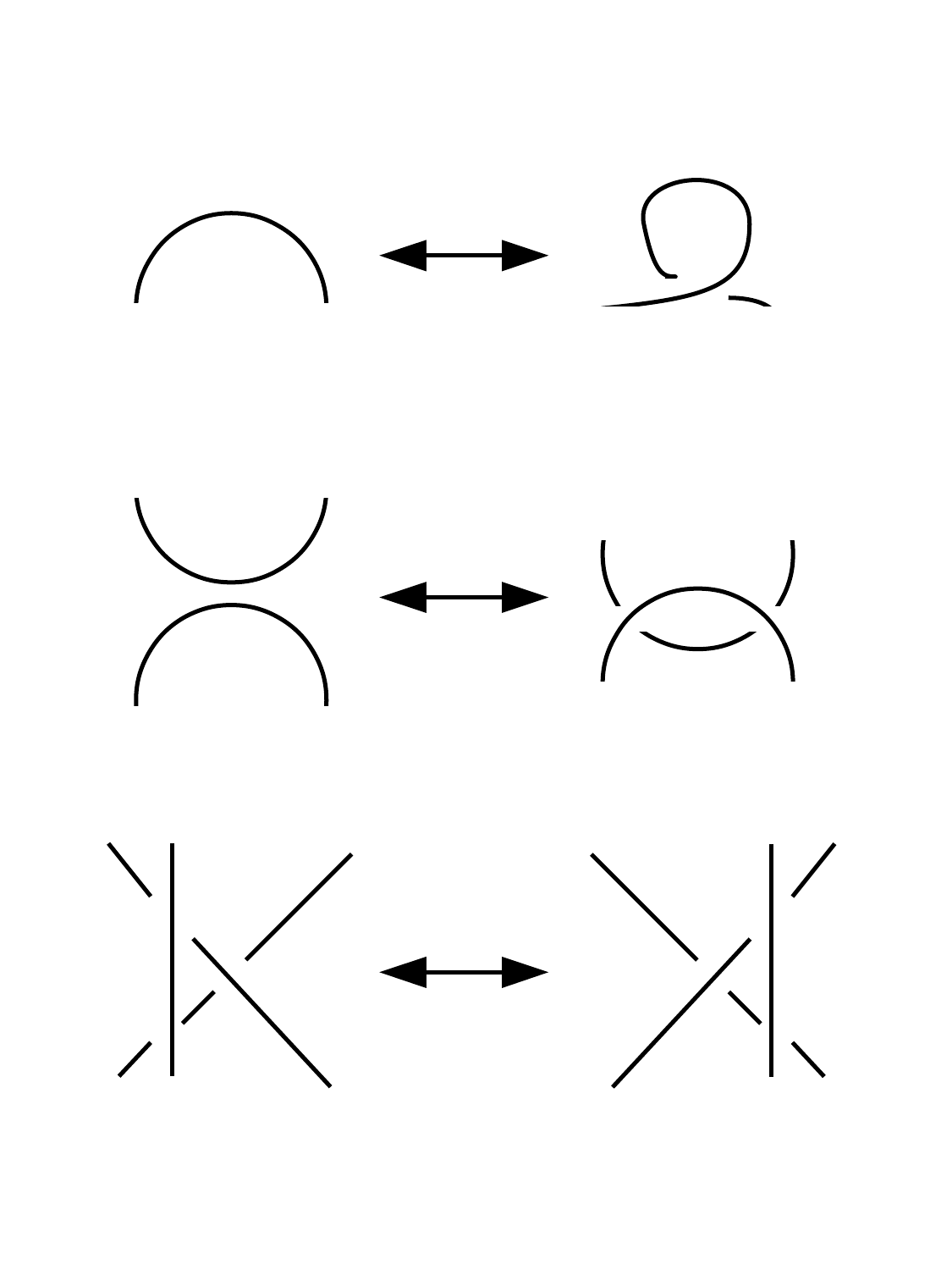}
		\caption{The three Reidemeister moves for diagrams of $1$-links.}
		\label{Reidemeister}
\end{figure}


Roseman, \cite{Roseman, Roseman1, Roseman2}, has generalized the diagram presentation of $1$-links to obtain an analogous presentation of $n$-links. In brief, it is possible to represent an arbitrary $n$-link $L$ in the manifold $F^{n+1}\times [0,1]$ by looking at the canonical projection of $L$ to $F$. Furthermore, it is possible to modify $L$ by an arbitrarily small isotopy such that the projection will be an immersion on an open dense subspace of $L$. By marking the double points of the projection with over and under information, Roseman obtains a generalization of link diagrams which is applicable to any dimension. Roseman also showed that any two link diagrams of ambient isotopic links will be related by a finite sequence of ambient isotopies of $F$ and certain local changes which leave the diagram unchanged outside a disk $D^{n+1}\subset F$. We will first give the full technical definitions necessary. Then we will examine some examples for $2$-links.

Let $L$ be an $n$-link in the manifold $F^{n+1}\times [0,1]$. If we wish to deal with links in $S^{n+2}$, we may form a diagram by isotoping the link to lie in some $\R^{n+1}\times [0,1]\subset S^{n+2}$; two links are isotopic in this $S^{n+1}\times [0,1]$ subspace iff they are isotopic in $S^{n+2}$.

The \emph{crossing set} of a link $L\subset F^{n+1}\times [0,1]$, denoted by $D^{*}$, is the closure of the set $\left\{x\in \pi(L), \left|\pi^{-1}(x)\right|\geq 2\right\}$. The \emph{double point set} $D$ is defined as $\pi^{-1}(D^{*})$. The \emph{branch set} $B$ is the subset of $L$ on which $\pi$ fails to be an immersion. The \emph{pure double point set} $D_{0}$ is defined to be the subset of $D$ consisting of those points $x\in L$ such that $\pi^{-1}(\pi(x))$ consists of exactly two points. We also define the \emph{(pure) overcrossing set} $D_{+}$ to be the set containing all $x\in D_{0}$ such that of the two $x, y\in\pi^{-1}(\pi(x))$, $\pi_I(y)<\pi_I(x)$, and the \emph{(pure) undercrossing set} $D_{-}=D_{0}-D_{+}$.

\begin{definition}\emph{(}\cite{Roseman1}\emph{)}
An immersion $\phi:L\rightarrow F$ is said to have \emph{normal crossings} if at any point of self-intersection, the pushforwards of the tangent spaces, $\phi_*:TL\rightarrow TF$, meet one another in general position.
\end{definition}

\begin{definition}

$L$ is in \emph{general position} (with respect to $\pi$) iff the following conditions hold.

\begin{enumerate}
	\item $B$ is a closed $(n-2)$ dimensional submanifold of $L$.

\item $D$ is a union of immersed and closed $(n-1)$-dimensional submanifolds of $L$ with normal crossings. The set of points where normal crossings occur will be labeled $N$ and called the \emph{self-crossing set of $D$}.

\item $B$ is a submanifold of $D$, and every $b\in B$ has an open disk neighborhood $U\subseteq D$, such that $U-B$ has two components $U_{0}$, $U_{1}$, which $\pi$ embeds into $F$ such that $\pi(U_{0})=\pi(U_{1})$.

\item $B$ meets $N$ transversely.

\item The restriction of $\pi$ to $B$ is an immersion of $B$ with normal crossings.

\item The crossing set of $\pi(B)$ is transverse to the crossing set of $\pi(L)$.

\end{enumerate}

\end{definition}

Note that this definition is taken from \cite{PR}, rather than Roseman's original version; the two definitions are equivalent.

This somewhat technical definition is important because of Theorem \ref{Rthm}, which says that we can always put an $n$-link into general position. Furthermore, Roseman has given a finite collection of moves, analogous to the Reidemeister moves, which relate any two diagrams of isotopic $n$-links, \cite{Roseman2}.

\begin{theorem}\emph{(}\cite{Roseman2}, \cite[Theorem 6.2, Lemma 7.1]{PR}\emph{)}\\
Every $n$-link in $F\times [0,1]$ can be isotoped into general position with respect to the projection $\pi :F\times [0,1]\rightarrow F$. \label{Rthm}
\end{theorem}

Note that in \cite{Roseman2}, Theorem \ref{Rthm} is proved in the case of links in $\R^{n+1}\times [0,1]$ only. However, by Lemma 7.1 in \cite{PR}, it is true for links in $F\times [0,1]$ as well.

\begin{definition}
A link diagram $d$ for an $n$-link $L\subset F\times [0,1]$, in general position with respect to $\pi$, is the projection $\pi(L)$, such that at each point of crossing set whose preimage consists of exactly two points, there are markings to determine which point in the preimage has the larger coordinate in $[0,1]$.
\end{definition}

Note that by the above theorem every $n$-link in $F\times [0,1]$ can be represented by a link diagram $D$ on $F$. It is easy to see that a link represented by $d$ is unique. 

Note that in the case $n=1$, this definition reduces to the usual definition of a link diagram. A $2$-link diagram will be a surface in a $3$-manifold $M$ with double point curves, which can meet in isolated triple points and terminate in isolated branch points. Over and under information can be given, as in the $n=1$ case, by breaking one of the surfaces in the neighborhood of a double point curve. A $2$-link diagram is for this reason also called a \emph{broken surface diagram}.

Roseman also shows, in \cite[Prop. 2.9]{Roseman2}, that, for any dimension $n$, two link diagrams for links which are ambient isotopic in $F\times [0,1]$ differ by a finite sequence of ambient isotopies of $\pi(L)\subset F$ and certain \emph{Roseman moves}. Each Roseman move changes the link diagram only within some open disk in $F\times [0,1].$ The set of Roseman moves for fixed $n$ is finite, although the number of moves increases with $n$. In the case $n=1$, the Roseman moves reduce to the standard Reidemeister moves, shown in Fig. \ref{Reidemeister}. In the case $n=2$, the Roseman moves can be given simple graphical representations in terms of broken surface diagrams. For higher $n$, Roseman gives a local model for each Roseman move, but does not give an explicit enumeration of moves. We will return to these representations in our discussion of diagrams for virtual links.

\begin{definition}
In a link diagram $d$, given by $\pi(L)$, for $L\subset F\times [0,1]$, the connected components of $\pi(L)-\pi(D)$ will be termed the sheets of the diagram.
\end{definition}

For $1$-links, a sheet is the same as a semi-arc. In addition, sometimes the term \emph{sheet} may be used to refer to the preimages in $L$ of the sheets in the diagram. When there is ambiguity, we will specify whether we are referring to a sheet in the diagram on $F$ or a sheet in $L$.

\begin{definition}
Let $L\subset F\times [0,1]$ be in general position with respect to $\pi:F\times [0,1]\rightarrow F$. The connected components of $(L-D)\cup D_+$ will be termed the \emph{faces} of the diagram. The images of these components under $\pi$ may also be called faces depending upon context.
\end{definition}

The faces of a $1$-link diagram are its arcs.

\section{Review of Virtual and Welded Links for $n=1$}

Virtual links were introduced by Kauffman, \cite{Kauff}, as a generalization of classical $1$-links. A link diagram gives rise to a combinatorial \emph{Gauss code}.

Since a link diagram is inherently a graph on a surface $F$, each link diagram gives rise to a \emph{Gauss code}. There are various ways of defining Gauss codes for links. All of them specify the set of arcs in the diagram, and, for each crossing, the two arcs which terminate at that crossing are specified, together with the order in which these crossings occur on their overcrossing arcs. Concretely, a Gauss diagram for $L=S^1\cup ...\cup S^1\xrightarrow{\imath} F\times I \xrightarrow{\pi} F$ is $L$ together with a collection of triples $(x,y,z)$, one for each vertex $v$ of the diagram $\pi\imath(L),$ where $x,y$ are preimages of $v$ under $\pi$ and $z=\pm 1$, depending on whether $y$ is over $x$ (i.e. $\pi_I(y)>\pi_I(x)$) or not.

However, not every Gauss code corresponds to a link diagram on $\R^2$. A Gauss code which corresponds to a link diagram on $S^2$ is termed \emph{realizable}. One method for defining virtual links is to consider arbitrary Gauss codes, modulo local changes corresponding to Reidemeister moves. These moves for Gauss codes are illustrated in Fig. \ref{gaussreid}. 

\begin{figure}
		\centering
			\includegraphics[scale=0.3]{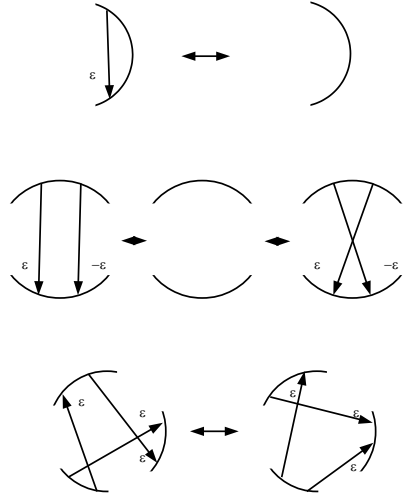}
		\caption{The effects of the three Reidemeister moves on a Gauss code. Here $\epsilon$ denotes the sign of the crossing and $-\epsilon$ denotes a crossing of the opposite sign.}
		\label{gaussreid}
\end{figure}

Such Gauss codes may also be represented using \emph{virtual link diagrams}. A virtual link diagram is a link diagram where we allow arcs to cross one another in \emph{virtual} crossings, indicated in the diagram by a crossing marked with a little
circle. All the Reidemeister moves are permitted on such diagrams, and any arc segment with only virtual crossings may be replaced by a different arc segment with only virtual crossings, provided the endpoints remain the same. This is equivalent to allowing the additional moves shown in Fig. \ref{vReidemeister}.

\begin{figure}
		\centering
			\includegraphics[scale=0.5]{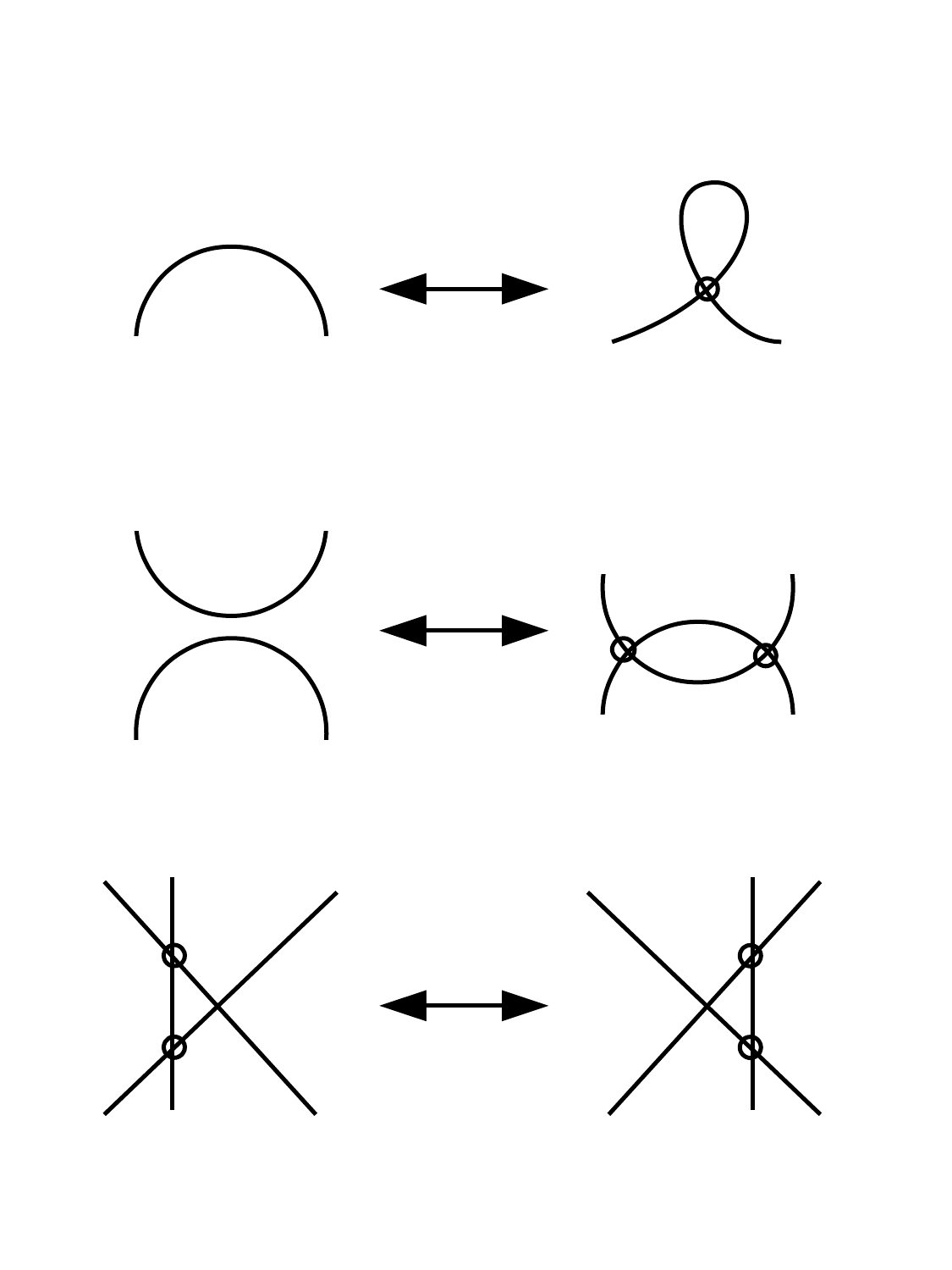}
		\caption{The additional virtual Reidemeister moves. Virtual crossings are indicated by crossings with small circles placed on them. The unmarked crossing in the third move can be either a virtual crossing or an arbitrary classical crossing.}
		\label{vReidemeister}
\end{figure}

Virtual links share many similarities with classical ones. For example, one can define virtual link groups and peripheral subgroups of virtual knot groups, cf. Sec. \ref{s-geom-def-virt}.
Furthermore, polynomial invariants of links, like the Alexander, Jones, and Homfly-pt polynomials, extend to virtual links.




Since the group and peripheral structure are invariants of virtual knots, by a result of Gorden and Luecke, \cite{GL} along with a theorem of Waldhausen, \cite{Wald}, two classical knots which are virtually equivalent must be classically equivalent as well. It can be shown, however, that not every virtual link is virtually equivalent to a classical link. We will discuss such an example in Section \ref{QaB}.

Welded links were first studied in the context of braid groups and automorphisms of quandles by Fenn, Rimanyi, and Rourke,	\cite{FRR}. For the $n=1$ case, welded links are virtual links considered modulo the so-called ``forbidden move'' (as it is a forbidden move for virtual links), shown in Fig. \ref{forbidden}. The link group, and the peripheral structure, can be extended to invariants of welded links. In particular the link group is computed from a welded link diagram using the same algorithm as for virtual or classical link diagrams. Consequently, just as in the virtual case, any two classical knots which are welded equivalent must be classically equivalent as well. There are examples of welded knots which are not welded equivalent to classical links, however. For example, there are welded knots whose fundamental groups are not infinite cyclic, but whose longitudes are trivial. Examples of such welded knots are easily constructed using a consequence of the work of Satoh, \cite{SS}. Satoh has shown that every ribbon embedding of a torus in $S^4$ can be represented by a welded knot, whose fundamental group and peripheral subgroup are isomorphic to the fundamental group and peripheral subgroup of the torus. Given any nontrivial ribbon embedding of $S^2$ into $S^4$, we may take the connected sum of this surface knot with a torus forming the boundary of a genus-$1$ handlebody in $S^4$. The resulting knotted torus will be ribbon and have a trivial longitude. It follows from Satoh's correspondence between welded knots and ribbon embeddings of tori in $S^4$ that a welded knot which represents this knotted torus will be nontrivial, but will have a trivial longitude. Diagrams of such welded knots will also provide examples of virtual knots which are not virtually equivalent to classical knots.

\begin{figure}
		\centering
			\includegraphics[scale=0.5]{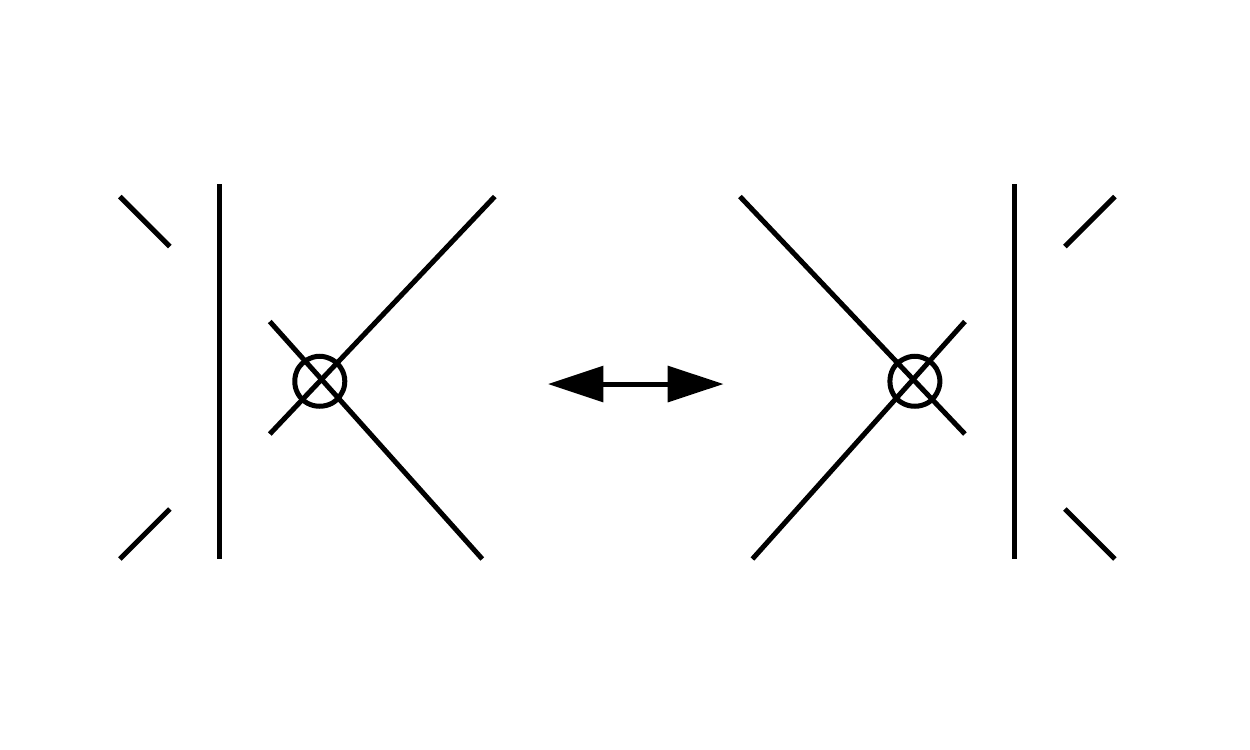}
		\caption{The ``forbidden move,'' which is not allowed for virtual $1$-links but is permitted for welded $1$-links.}
		\label{forbidden}
\end{figure}

As we will see, the fact that every welded $1$-link diagram is also a virtual $1$-link diagram is unique to the case $n=1$; for larger $n$, the geometrically-motivated definition of a welded link, which we will give, does not reduce to the notion of a virtual link with an additional move permitted.

\section{Ribbon Links}


Because we will be making reference to ribbon links and their ribbon presentations in further sections, we give here a brief review of these concepts. See \cite{CKS} for further discussion of ribbon links in arbitrary dimensions.

\begin{remark}
Throughout, we will use the notation $D^n$ to represent a closed $n$-ball.
\end{remark}

Let $M$ be an $(n+2)$-manifold. Consider a pair $(B, H)$, consisting of a set of disjoint oriented $(n+1)$-disks $B=\{B_{i}:D^{n+1}\rightarrow M\}$ and a collection $H$ of oriented $1$-handles $H_{j}: D^1\times D^n\rightarrow M$, with $H_j(\{ 0,1\}\times D^n)$ embedded in $\cup_{B_i\in B}B_i(\partial D^{n+1})$, such that the orientations agree, and such that for each $i, j$, the components of $H_j((0,1)\times D^n)\cap B_i(D^{n+1})$ are a finite collection of $n$-disks embedded in $H_j((0,1)\times D^n)$, such that each component separates $H_j(\{0\}\times D^n)$ from $H_j(\{1\}\times D^n)$. This implies that $H_j((0,1)\times D^n)\cap B_i(D^{n+1})$ is contained in the interior of the image of $B_i$; such intersections between a handle and a base are called \emph{ribbon intersections} or \emph{ribbon singularities}. A handle can intersect any number of bases in ribbon intersections, and it can intersect the same base multiple times (or zero). We will in general identify the maps $B_i$ and $H_j$ with their images.
The disks in $B$ are called \emph{bases}, while the elements of $H$ are called \emph{(fusion) handles} or \emph{bands}.
Let $K$ be the link defined as $K=({\cup}_{B_i\in B}(\partial B_{i}-{\cup}_{H_j\in H}H_{j}))\cup({\cup}_{H_j\in H}H_{j}(D^1\times\partial D^n))$. Then $K$ is called a \emph{ribbon link}, and the pair $(B, H)$ are a \emph{ribbon presentation} for $K$. We identify two ribbon presentations, $(B, H)$ and $(B', H')$, of $K$ if there is a path between them in the space of ribbon presentations of $K.$ This is equivalent to requiring that $|B|=|B'|$, $|H|=|H'|$, and that there exists an ambient isotopy $f_t$ of $M$, $t\in [0,1]$, such that $f_0$ is the identity, and for all $B_i\in B, H_j\in H, B'_i\in B', H'_j\in H'$, $f_1\circ B_i=B'_i$, and $f_1\circ H_j=H'_j$.

The subspace $R=(\cup_{B_i\in B}B_{i})\cup(\cup_{H_j\in H}H_{j})$ is called a \emph{ribbon solid} for the ribbon presentation or for the link. Note that a given $R$ may be a ribbon solid for multiple ribbon presentations. If an $n$-link $K$ has a ribbon presentation then it is called a \emph{ribbon link}.

Observe that since we require that the handles be glued to the bases in an orientation-preserving manner, $K$ will be orientable. Note also that for ribbon $1$-links, the components of a ribbon link may not be ribbon knots. However, for $n$-links with $n\geq 2$, each component of a ribbon link will automatically be a ribbon knot.

Given a ribbon link $K$, there are in general multiple ribbon presentations for $K$, making it convenient to consider notions of equivalence for ribbon presentations. We will, for the moment, restrict our attention to knots; the generalization to links is straightforward. Let $K$ be a ribbon knot with ribbon presentations $(B, H)$ and $(B, H')$. These presentations are called \emph{simply equivalent} if $|H|=|H'|$, $|B|=|B'|$, and there exists an isotopy $f_t$, $t\in [0,1]$, of $M$, such that $f_0$ is the identity, and for all $B_i\in B, H_j\in H, B'_i\in B', H'_j\in H'$, $f_1\circ B_i|_{\partial D^{n+1}}=B'_i|_{\partial D^{n+1}}$, and $f_1\circ H_j=H'_j$.

Note that simple equivalence may create new intersections of handles with the interiors of bases, as illustrated in Fig. \ref{ribbonstable2}.

\begin{figure}
		\centering
			\includegraphics[scale=0.3]{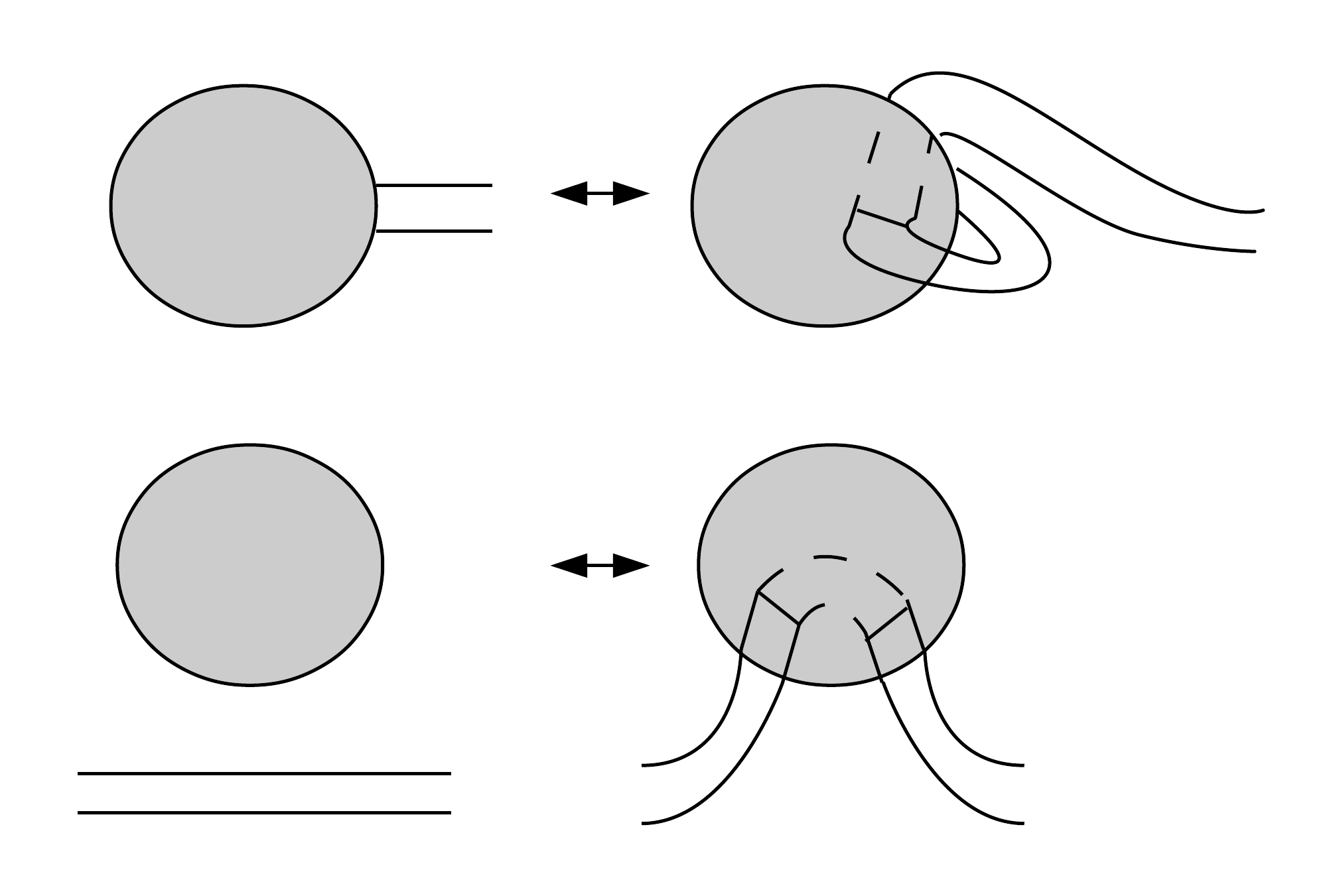}
		\caption{Two simply equivalent ribbon presentations with distinct ribbon intersections along a single handle. Here, the bases are indicated with shading; ribbon intersections are shown by thickened lines.}
		\label{ribbonstable2}
\end{figure}

\begin{lemma}
For $n\geq 2$ a ribbon presentation $(B,H)$ in a simply-connected manifold $M$ is determined uniquely up to simple equivalence by specifying the following ribbon data:\\
(a) the cardinality of $B$ and,\\
(b) for each $H_{j}\in H$, a sequence of bases, $B_{j1},...,B_{jk},$
through which the handle $H_j$ passes through, up to reversal of that sequence.
\end{lemma}

\bpr Clearly there is a ribbon presentation fulfilling every possible ribbon data as above.
We need to show that such ribbon presentation is unique, up to simple equivalence. Note that for for each $j$, the sequence $B_{j1},...,B_{jk},$ determines the homotopy class of the handle $H_j$ in $M-\cup_i \p B_i$ uniquely. Since $n\geq 2$ and the core of the handle is $1$-dimensional, the homotopy class of the handle determines its isotopy class as well.
\epr

In $n=1$ this is of course not the case, since the homotopy class of the handle does not determine its isotopy class in $M-\cup_i \p B_i$. In addition, one must specify a cyclic ordering  of handles glued to a base for $n=1$.

There is a broader notion of equivalence between ribbon presentations called \emph{stable equivalence}, defined by Nakanishi, \cite{Naka}. This is the equivalence relation generated by simple equivalence together with the following three moves and their inverses, which are illustrated in Fig. \ref{ribbonstable}:

\begin{enumerate}
	\item Add a new base $B_{0}$ with a handle $H_{0}$ starting on $B_{0}$ and terminating on any other base, without passing through any bases.
	\item \emph{Handle pass}: Isotope $H_{j}$ through $H_{j'}([0,1]\times (D^n-\partial D^n))$, leaving $H_j|_{\{ 0,1\}\times D^n}$ fixed.
	\item \emph{Handle slide}: If $H_{j}$ connects $B_{i}$ with $B_{i'}$, and $H_{j'}$ has an end on $B_{i}$, we may slide that end along $H_{j}$ over to $B_{i'}$.
\end{enumerate}

\begin{figure}
		\centering
			\includegraphics[scale=0.6]{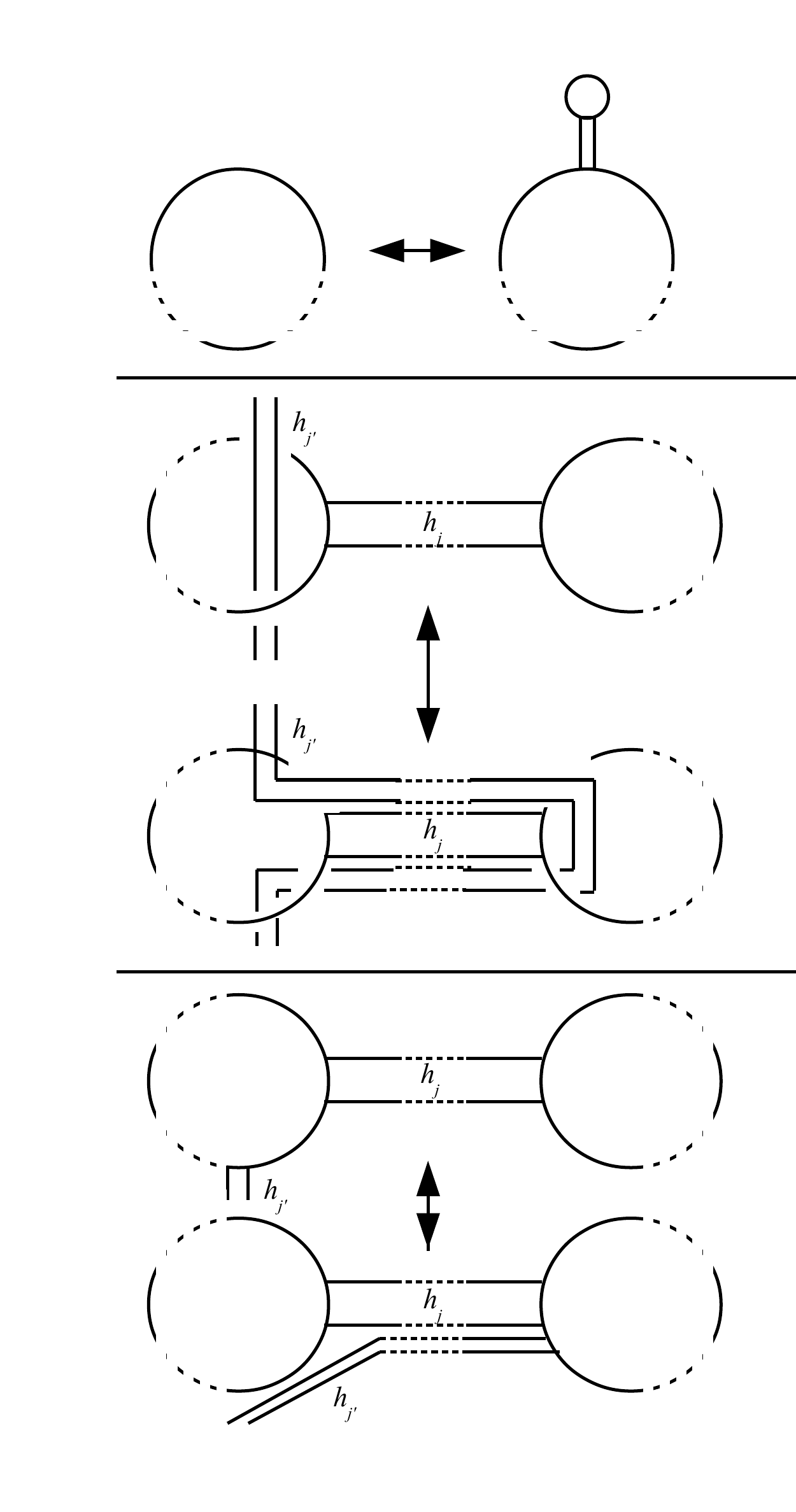}
		\caption{The three moves permitted for ribbon presentations under stable equivalence.}
		\label{ribbonstable}
\end{figure}

Observe that for a knotted $n$-sphere, $|H|=|B|-1$ by Euler characteristic arguments. It is easy to check that for $n\geq 2$, a ribbon $n$-knot $K$ will have the homeomorphism type of an $n$-sphere with $|H|-|B|+1$ $1$-handles attached.

Given a ribbon $n$-knot $K$, it is an open question whether any two ribbon presentations for $K$ are related by stable equivalence for $n\geq 2$. For the case $n=1$, Nakanishi and Nakagawa provide a family of knots, each of which admits multiple ribbon presentations which are not stably equivalent, \cite{NakaNaka, Naka}.

\section{Quandles and Biquandles}\label{QaB}

Quandles and biquandles are nonassociative algebraic structures which provide invariants for classical and virtual links. As we will show, they also provide invariants for our generalization of virtual links to higher dimensions. Since we will be referring to these structures repeatedly, we list their definitions here for reference.

A \emph{quandle}, \cite{DJ}, (also called a \emph{distributive groupoid} in \cite{Mat}) $(Q, *)$ consists of a set $Q$ and a binary operation $*$ on $Q$ such that for all $a, b, c \in Q$, the following equalities hold.

\begin{enumerate}
	\item $a*a=a$.

\item There is a unique $x\in Q$ such that $x*b=a$.

\item $(a*b)*c=(a*c)*(b*c)$.

\end{enumerate}
When the operation $*$ is implied, we may denote the quandle by $Q$, suppressing the operation.

\begin{example}
For any group $G$, $(G,*)$ is a quandle, for $a*b=a^{b},$ (i.e. $bab^{-1}$). We call $(G,*)$ the group quandle of the group $G.$
\end{example}

It is easily checked that this defines a forgetful functor from the category of quandles to the category of groups.

In general, it is common to denote the $*$ operation in any quandle $Q$ by using the conjugation notation, that is, for any quandle we may define $a^{b}=a*b$. Analogously, we will denote the element $x$ stipulated by property (2) by $a^{\bar{b}}$, i.e. $(a^b)^{\bar{b}}=a$. By convention, we interpret $a^{bc}=(a^b)^c$. In Section \ref{geoqu} we review Joyce's geometric definition of the fundamental quandle of the complement of any $n$-link.

While a quandle can be thought as a set with two binary operations, $a^b$ and $a^{\bar b},$ a \emph{(strong) biquandle}, \cite{BQ1, biq, Carrell}, is a set $B$ together with four binary operations, often denoted using exponent notation as $a^{b}$, $a_{b}$, $a^{\bar{b}}$, $a_{\bar{b}}$. As for quandles, these operations are assumed to be associated from left to right unless otherwise specified with parentheses, for example, $a^{bc}=(a^b)^c$. These operations are required to satisfy the following properties:

\begin{enumerate}
\item For any fixed $b$, the functions sending $b$ to $a^{b}$, $a_{b}$, $a^{\bar{b}}$, $a_{\bar{b}}$ are bijective functions on $B$ with an argument $a$ (this implies all four operations have right inverses).

\item For any $a, b, c \in B$, $c=a_{c}$ iff $a=c^{a}$, and $b=a^{\bar{b}}$ iff $a=b_{\bar{a}}$.

\item The following equalities hold for all $a, b, c \in B$:
$a=a^{b\overline{b_{a}}}$,
$b=b_{a\overline{a^{b}}}$,
$a=a^{\overline{b}b_{\overline{a}}}$,
$b=b_{\overline{a}a^{\overline{b}}}$,
$a^{bc}=a^{c_{b}b^{c}}$,
$c_{ba}=c_{a^{b}b_{a}}$,
$(b_{a})^{c_{a^{b}}}=(b^{c})_{a^{c_{b}}}$,
$(b_{\overline{a}})^{\overline{c_{\overline{a^{\overline{b}}}}}}=(b^{\overline{c}})_{\overline{a^{\overline{c_{\overline{b}}}}}}$,
$a^{\overline{bc}}=a^{\overline{c_{\overline{b}}b^{\overline{c}}}}$,
$c_{\overline{ba}}=c_{\overline{a^{\overline{b}}b_{\overline{a}}}}$.
\end{enumerate}



Every quandle $(Q, *)$ can be made into a biquandle by defining the biquandle operations to be $a^b=a*b$, $a^{\overline{b}}=a*^{-1}b$, $a_b=a_{\overline{b}}=a$. Another family of examples may be constructed by taking $B$ to be a module over $\Z [s, t, s^{-1},t^{-1}]$. Then we define biquandle operations on $B$ to be $a^b=ta+(1-st)b$, $a^{\overline{b}}=t^{-1}a+(1-s^{-1}t^{-1})b$, $a_b=sa$, and $a_{\overline{b}}=s^{-1}a$. Then $B$ is a biquandle called an \emph{Alexander biquandle}, \cite{NelLam}.

Although the axioms for a biquandle appear somewhat abstract, they are all motivated by considering labelings of semi-arcs in diagrams for virtual links or labelings of sheets of broken surface diagrams for surfaces embedded in $S^{4}$, and then requiring that the resulting algebraic structure must be invariant under the Reidemeister or $2$-dimensional Roseman moves.

We may also define \emph{weak} biquandles by dropping the first axiom (i.e. the axiom that each biquandle operation has a right inverse). However, for our purposes, we will always use the term ``biquandle'' to indicate a strong biquandle, unless otherwise noted.

There are also other conventions for denoting the biquandle operations; we have followed Carrell's notation, \cite{Carrell}, since that source discusses applications of biquandles to knotted surfaces. Biquandles can also be defined as sets $Q$ with two right-invertible operations which satisfy the following identities (see \cite{Stan} for a discussion of these): $a_{(bc)}=a_{c^{b}b_c}$, $a^{(bc)}=a^{c_{b}b^c}$, $(a_b)^{c_{b^a}}=(a^c)_{b^{c_a}}$, $(a_{\overline{a}})^{\overline{a_{\overline{a}}}}=a$.

Biquandles have been found to be useful for distinguishing virtual $1$-knots that are indistinguishable by quandles. For example, the Kishino knot, Fig. \ref{KI}, is a virtual knot whose quandle is isomorphic to the quandle of the unknot. However, it can be distinguished from the unknot using the biquandle, \cite{Kauff2, BF}.

\begin{figure}
		\centering
			\includegraphics[scale=0.5]{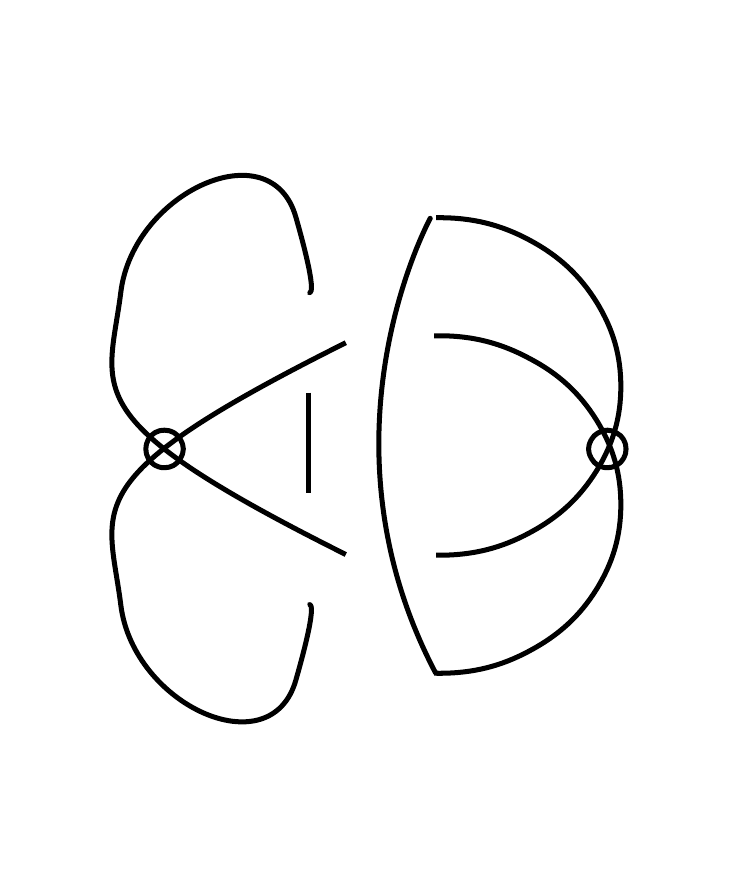}
		\caption{A virtual knot diagram of the Kishino knot.}
		\label{KI}
\end{figure}
%
\section{The Link Quandle and Group in Arbitrary Dimensions}
\label{geoqu}
%

Quandles as knot invariants were introduced in \cite{DJ, Mat}.
In this section we discuss the geometric definition of the fundamental quandle of an $n$-knot, originally given by Joyce in \cite{DJ}. In that paper, it was shown that the quandle of a classical knot is determined by its knot group together with the specification of a meridian and the peripheral subgroup. We will show that this is true for arbitrary (orientable) knots in spheres, as well as for virtual $n$-knots.


Let $L$ be a link in $M$, with $L$ and $M$ assumed to be orientable, and let $N(L)$ be an open tubular neighborhood of $L$. Then $M-N(L)$ is a manifold with boundary containing $\partial \overline{N(L)}$. The fundamental group of this space is called the \emph{link group}. We will follow the convention that the product of two homotopy classes of paths goes from right to left, i.e. $\gamma' \gamma$ is the homotopy class of a path which follows $\gamma$ and then $\gamma'$.
We call the image $P$ of the homomorphism induced
by the injection $\partial \overline{N(L)}\rightarrow M-N(L)$ the \emph{peripheral} subgroup of $\pi_1(M-L).$ Note that $P$ is defined only up to conjugation;
any of its conjugates are also peripheral subgroups.

A meridian $m$ of $L$ is the boundary of the disk fibre of the tubular neighborhood of $L$ considered as a disk bundle over $L$. Note that the orientations of $M$ and of $L$ determine an orientation of $m$. A meridian is defined uniquely up to conjugation.

\begin{figure}[htbp]
\begin{centering}
\includegraphics[scale=.8]{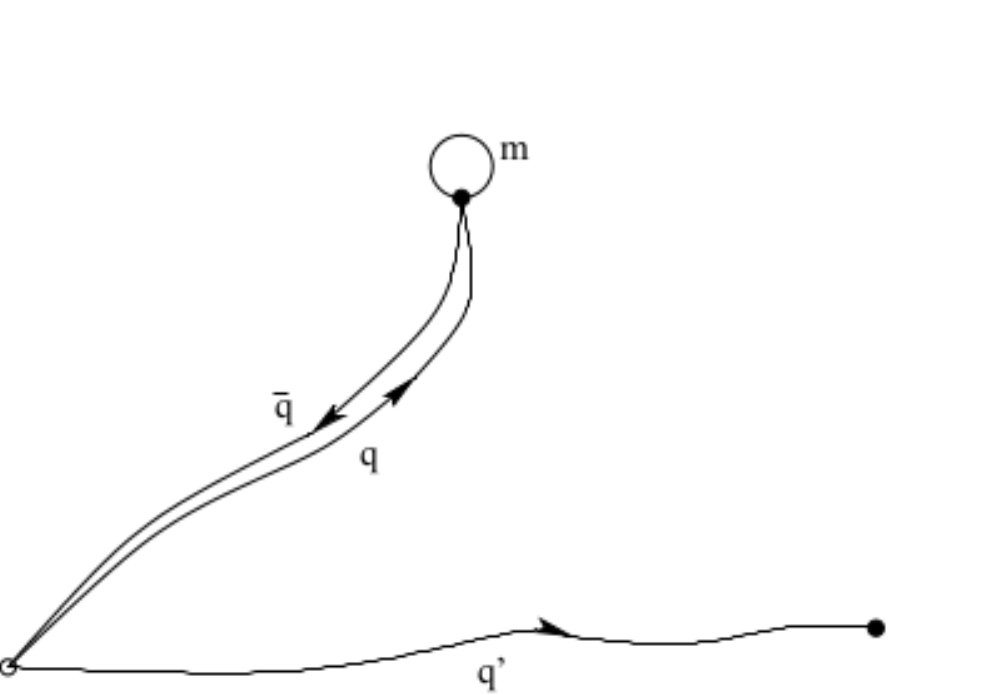}
\caption{The quandle operation for the fundamental quandle of a link $L$ in any space $M$. The
open circle indicates the basepoint of $M$, while the black circles represent points on the boundary of a tubular neighborhood of $L$.}
\label{qop}
\end{centering}
\end{figure}

The definition of the quandle, and of the fundamental group of a link complement, requires concatenation of paths. We will use the following convention for path concatenation: given two paths $\gamma, \gamma'$, we will write $\gamma' \gamma$ to indicate the path that follows $\gamma$ first, and then follows $\gamma'$. This convention makes it easier to discuss the action of the link group on the quandle.

We define the quandle $Q(M,L)$ associated to the pair $(M,L)$ and a chosen base point $b\in M-N(L)$ as follows. The elements of $Q(M,L)$ are homotopy classes of paths in $M-N(L)$ that start at $b$ and end in $\partial \overline{N(L)}$, where homotopies are required to preserve these two conditions. The quandle operation, $[q]^{[q']}$, between two elements $[q], [q']\in Q(M,L)$ is defined as the homotopy class of $q\overline{q'}mq'$, where the overbar indicates that the path is followed in reverse and $m$ is the meridian of $L$ passing through the endpoint of $q$, cf.  \ref{qop}. It is a straightforward exercise to show that this binary operation is well-defined and it satisfies the definition of a quandle operation. We call $Q(M,L)$ the \emph{link quandle} of $L$ (or the knot quandle if $L$ is a knot).

Note that there is a right action of $\pi_{1}(M-N(L))$ on $Q(M,L)$: $[q]g$ is the equivalence class including the homotopy class of the path $qg$.

\subsection{Quandles From Groups}
Given a group $G$, a subgroup $P$, and an element $m$ in the center, $Z(P),$ of $P$, Joyce defines a quandle operation on the set of right cosets of $P$, $P\backslash G$, as follows: $Pg^{Ph}=P(gh^{-1}mh)$. This operation is well-defined, because if $h'\in Ph$, then $h'=ph$ for some $p\in P$ and
\begin{equation}
  h'^{-1}mh'=h^{-1}p^{-1}mph=h^{-1}mh,
\end{equation}
since $m\in Z(P)$.
We will denote a quandle constructed in this way
by $(P\backslash G,m)$, or simply $P\backslash G$ when there is a canonical choice for $m$, for example,
when $m$ is a meridian and $P$ a peripheral subgroup of a knot.

\subsection{Group Actions on Quandles}
The results in this section are due to Joyce, \cite{DJ}.


For any link $L$ in $M$, $\pi _{1}(M-L)$ acts on $Q(M,L)$ by setting $[q]g$ to be the homotopy class of the path $qg$.

Consider now a knot $K\subset M$ and a path $m_Q$ connecting a fixed base point $b\in M-N(K)$ with $\p N(K).$  (Clearly, $m_Q\in Q(M,K).$)
Note that $m_Q$ defines a peripheral subgroup $P$ of $\pi_1(M-N(K),b).$
Its elements are of the form $\overline{m_Q}\alpha m_Q,$ where $\alpha$ is any loop in $\p N(K)$ based at the endpoint of $m_Q.$ In particular, $m_Q$ determines a meridian $m\in P.$


\begin{lemma}\label{G-act-on-Q}
For any knot $K\subset M$,\\
(1) the $\pi_{1}(M-K)$ action on $Q(M,K)$ is transitive, i.e. every element of $Q(M,K)$ is equal to $m_{Q}g$ for some $g\in \pi_{1}(M-K)$.\\
(2) $P$ is the stabilizer of the $\pi_{1}(M-K)$ action on $m_Q\in Q(M,K),$
i.e. $m_Qg=m_Q$ iff $g\in P.$
\end{lemma}

\bpr (1) Choose representatives of $q$ with an endpoint coinciding endpoints on $\partial N(K)$. Then $q=m_Qg,$ for $g\in \pi_{1}(M-N(K))$ given by
the path $m_{Q}\overline{m_{Q}}q$.\\
(2) Since every $g\in P$ is of the form $\overline{m_Q}\alpha m_Q$,
$m_Qg=m_Q\overline{m_Q}\alpha m_Q=\alpha m_Q$ which can be homotoped along $\alpha$ to $m_Q$.
Conversely, suppose that $m_Qg=m_Q.$ The homotopy transforming $m_Qg$ into $m_Q$ moves the endpoint of $m_Qg$ along a closed loop $\alpha$ in $\p N(K).$
That means that $m_Qg$ and $\alpha m_Q$ are homotopic in $M-N(K)$ with their endpoints fixed. Consequently, $g=\overline m_Q \alpha m_Q,$ i.e. $g\in P.$

\epr

Note that Lemma \ref{G-act-on-Q}(2) implies that the stabilizer of $m_{Q}g$ is $g^{-1}Pg$.


Consider the peripheral subgroup $P$ and the meridian $m$ determined by a path $m_Q$ from the base point $b$ in $M-N(K)$ to $\p N(K)$, as above. Then the map
$\Psi:Q(M, K)\rightarrow (P\backslash \pi_{1}(M-K,b),m)$ sending $m_{Q}g$ to $Pg$ is well defined by Lemma \ref{G-act-on-Q}(2).

\begin{theorem}\label{main}
$\Psi:Q(M, K)\rightarrow (P\backslash \pi_{1}(M-K,b),m)$  is an isomorphism of quandles. Consequently, $Q(M,K)$ is determined by $\pi_{1}(M-K)$ with its peripheral structure.
\end{theorem}

\bpr

$\Psi$ is $1$-$1$ by Lemma \ref{G-act-on-Q}(2), as well as onto, by Lemma \ref{G-act-on-Q}(1). To show that $\Psi$ is a quandle homomorphism, observe that
\begin{equation}
  (m_{Q}g)^{m_{Q}h}=m_{Q}g\overline{h}\overline{m_{Q}}mm_{Q}h=m_{Q}g\overline{h}mh.
\end{equation}
It follows that
\begin{equation}
  \Psi((m_{Q}g)^{m_{Q}h})=\Psi(m_{Q}g\overline{h}mh)=Pg\overline{h}mh=Pg^{Ph}=\Psi(m_Qg)^{\Psi(m_Qh)}.
\end{equation}
Therefore, $\Psi$ is a bijective quandle morphism.
\epr

\subsection{Remarks}\label{JR}
Joyce has also proved a theorem for classical knots
stating that the triple $(\pi_{1}(M-K), P, m)$ can be reconstructed from $Q(M,K)$. We generalize his result here.

Let $L$ be an $n$-knot in $F\times [0,1]$ in general position with respect to $\pi: F\times [0,1]\to F.$ Denote its diagram by $d.$
Let $G(d)$ be a group with the following presentation:
its generators are the faces of $d$. Let $D_0=\left\{x\in \pi(L), |\pi^{-1}(x)|=2\right\}$ be the pure double point set. Then for each connected component $\gamma$ of $D_0$ consider
the relation $x=yzy^{-1},$ where $x, y, z$ are the three faces meeting at $\gamma$ with $y$ being the overcrossing face, and with the normal vector to $y$ pointing toward $x$, as illustrated in Fig. \ref{Wirtr}.

\begin{figure}[htbp]
\begin{centering}
\includegraphics[scale=0.5]{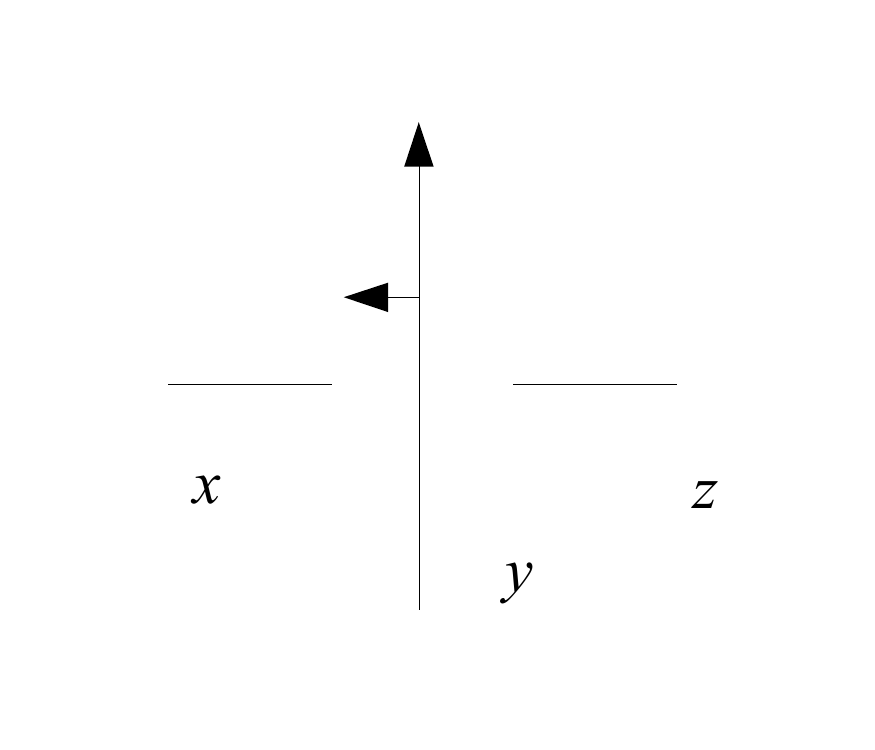}
\caption{A double point curve as shown yields relations $x=z^y$ on the group or quandle.}
\label{Wirtr}
\end{centering}
\end{figure}

Similarly we associate with $d$ a quandle $Q(d)$. Its generators are faces of $\pi(L).$ The relations between faces are $x=z^y,$ where $x,y,z$ are as above. Note that, in general, this quandle is not isomorphic to the group quandle of the knot group of $L$. For example, the knot group and its group quandle do not distinguish the square knot from the granny knot, but the knot quandles for these two knots are distinct.

\begin{theorem}\emph{(}\cite[Theorem 2]{KamaWirt}\emph{)}\label{KW}
For any knot $F^{n+1}$ simply connected, and $K\subset F^{n+1}\times I$ in general position with respect to $\pi$, with the knot diagram $d$, $G(d)$ and $Q(d)$ are isomorphic to the knot group and knot quandle of $K$.
\end{theorem}

This theorem is proved by an iterative application of the Van Kampen theorem. Note that in \cite{KamaWirt} only groups are considered, but the proof for quandles is analogous. The Van Kampen theorem holds for quandles, by \cite[Theorem 13.1]{DJ}.

\begin{theorem}\label{hid}
For any $n$-knot $K$ in $D^{n+2}$, the triple $(\pi_1(D^{n+2}-K), P, m)$ is uniquely determined (up to an isomorphism) by the knot quandle $Q(D^{n+2}, K)$.
\end{theorem}

\bpr
For any quandle $Q$, we may define the group $Adconj(Q)$ to be the group generated by the elements of $Q$, whose relations are exactly the quandle relations, with the quandle operation interpreted as conjugation in the group. By Theorem \ref{KW}, $Q(D^{n+2}, K)$ has a presentation whose generators correspond to the faces of a knot diagram of $K$, and whose relations are the Wirtinger relations corresponding to each meeting of faces in a double point set. It is then easy to check that $Adconj(Q(D^{n+2}, K))$ has a presentation as a group whose generators correspond to the faces of a knot diagram of $K$, with relations being exactly the Wirtinger relations from each double point component. It follows that $Adconj(Q)$ is canonically isomorphic to $\pi_{1}(D^{n+2}-K)$ (up to a choice of a basepoint). There is also a canonical map from the elements of $Q$ to the elements of $Adconj(Q)$, sending a word in the generators of $Q$ to the corresponding word in the generators of $Adconj(Q)$, with the quandle operation again interpreted as group conjugation.
For any $q\in Q$ denote the corresponding element in $Adconj(Q)$ by $\hat{q}$.
Then $m\in Adconj(Q(D^{n+2}, K))$ can be assumed to be any of the generators corresponding to a face in the knot diagram. In general,
$Adconj(Q(D^{n+2}, K))$ acts on $Q(D^{n+2}, K)$ by $q\hat{q_{0}}...\hat{q_{n}}=(...(q^{q_{0}})^{q_{1}}...)^{q_{n}}$. To see that this action is well-defined, note that if an element
of $Adconj(Q(D^{n+2}, K))$ has two presentations, $\hat{q_{0}}...\hat{q_{n}}$ and $\hat{p_{0}}...\hat{p_{m}}$,
this implies that those products of Wirtinger generators are homotopic rel basepoints, and
this homotopy passes to a homotopy of $(...(q^{q_{0}})^{q_{1}}...)^{q_{n}}$ to $(...(q^{p_{0}})^{p_{1}}...)^{p_{m}}$.
It is then straightforward to check that this action of $Adconj(Q(D^{n+2}, K))$ is just the geometrically defined action of $\pi_{1}(D^{n+2}-K)$ on $Q(D^{n+2}, K)$.
Then $P$ will be the stabilizer of $m\in Q(D^{n+2}, K)$ under this action, and so we have constructed
$(\pi_1(D^{n+2}-K), P, m)$. Note that $m$ is unique (up to conjugation) and $P$ is determined by $m$. Hence, this triple is unique up to an isomorphism.
\epr

Quandle invariants are particularly simple for knotted $n$-spheres, $n\geq 2$, since their peripheral groups are cyclic and, hence $m$, is unique. (The choice between $m$ and $m^{-1}$ is determined by the orientations of $M$ and of $K$). It follows that for classical knots, the quandle does not contain much more information than the fundamental group. On the other hand, when $\pi_1(\p N(K))$ is a more complicated subgroup of $\pi_1(M-N(K))$, the knot quandle may be able to capture more information than $\pi_1(M-N(K))$.

Eisermann, \cite{ME}, has in fact shown that the quandle cocycle invariants of classical knots are a specialization of certain colorings of their fundamental
groups. His construction makes use of the full peripheral structure of the classical knot
(that is, the longitude and the meridian),
and thus does not generalize immediately to higher dimensions. However, in light
of our result here, we pose the question of whether a similar construction might
not be possible in higher dimensions.

%
\section{Geometric Definition of Virtual Links in Any Dimension}
\label{s-geom-def-virt}
%

Kauffman defined virtual links as a combinatorial generalizations of classical link diagrams or, alternatively, as the Gauss codes of classical links diagrams. However, there is a geometric definition of virtual links as well. We will follow essentially the treatment given in \cite{KamaKama} and \cite{StableEq} for this approach. Note that while Carter et al. work with link diagrams on surfaces, we work with links in thickened surfaces; By Theorem \ref{Rthm} and the remarks below it, these approaches are equivalent. As we will see, the definition of \cite{KamaKama} and \cite{StableEq} can be extended to higher dimensions in a purely geometric manner, without reference to any combinatorial constructions.

Let $V_{n}$ be the set of pairs $(F\times [0,1], L)$ where $F$ is a compact $(n+1)$-manifold (with possibly non-empty boundary), $F\times [0,1]$ is given the structure of a trivial $[0,1]$-bundle over $F$, and where $L$ is an $n$-manifold embedded in the interior of $F\times [0,1]$ such that $L$ meets every component of $F\times [0,1]$. The $[0,1]$-bundle structure on $F\times [0,1]$ will always be implied to be the canonical bundle coming from projection $\pi:F\times [0,1]\rightarrow F$ when we refer to elements of $V_n$ in this way. For convenience we will either simply write $(F\times [0,1], L)$, with the bundle structure implied, or, if using the notation $(M, L)\in V_n$, we will assume $M$ is identified with some trivial bundle $F\times [0,1]$, with the bundle structure coming from projection onto the first factor of the Cartesian product.
(Our theory will not be affected by the fact that $F\times [0,1]$ may have corners.) 
Define a relation $\sim$ on $V_{n}$ by the condition $(F_{1}\times [0,1], L_{1})\sim(F_{2}\times [0,1], L_{2})$ iff there exists an embedding $f:F_1\rightarrow F_2$ such that $f\times id_{[0,1]}(L_1)=L_2$. We define an equivalence relation $\cong$ on $V_n$, $(F_{1}\times [0,1], L_{1})\cong (F_{2}\times [0,1], L_{2})$, to be the equivalence relation generated by $\sim$ together with smooth isotopy of the link. When $(F_{1}\times [0,1], L_{1}) \cong (F_{2}\times [0,1], L_{2})$, we say they are \emph{virtually equivalent}.


Consider first the case $n=1$. Note that any link diagram in a surface $F$ defines a Gauss code and, consequently, a virtual link. 
Since this construction is invariant under Reidemeister moves, it descends to a map $f: V_1\to$ \{Kauffman's virtual links\}. 
Furthermore, since $f(F_1\times [0,1],L_1)=f(F_2\times [0,1],L_2)$ for $(F_1\times [0,1],L_1)\cong (F_2\times [0,1],L_2),$
it factors to $f: V_1/\cong\ \to \{\text{virtual links}\}.$

\begin{definition}
An element $(F\times [0,1], L)$ of $V_{n}$, with $L$ in general position with respect to $\pi$, is an \emph{abstract $n$-link} iff $\pi(L)$ is a deformation retract of $F\times [0,1]$.
\end{definition}

We now wish to define an inverse map to $f$. Given a virtual link diagram $C$, we may construct an abstract link $(F\times [0,1],L)\in V_{1}$ using the following construction. Let $F\times [0,1]'$ be a neighborhood of the graph $C$ in $S^{2}$. Replace the virtual crossings by changing the thickened ``+'' shape in a neighborhood of each virtual crossing into two strips with an arc along each of them, as shown in Fig. \ref{arch}. The result is a $2$-manifold $F$ with a link diagram on it, which therefore defines a link in $F\times [0,1]$. Let us denote its equivalence class in $V_1/\cong$ by $g(C).$ If two virtual link diagrams $C$ and $C'$ are equivalent under the virtual Reidemeister moves, then it is easily checked that $g(C)\cong g(C')$. We may therefore regard the map $g$ as a map from equivalence classes of virtual link diagrams to equivalence classes of virtual $1$-links.

\begin{figure}
		\centering
			\includegraphics[scale=0.5]{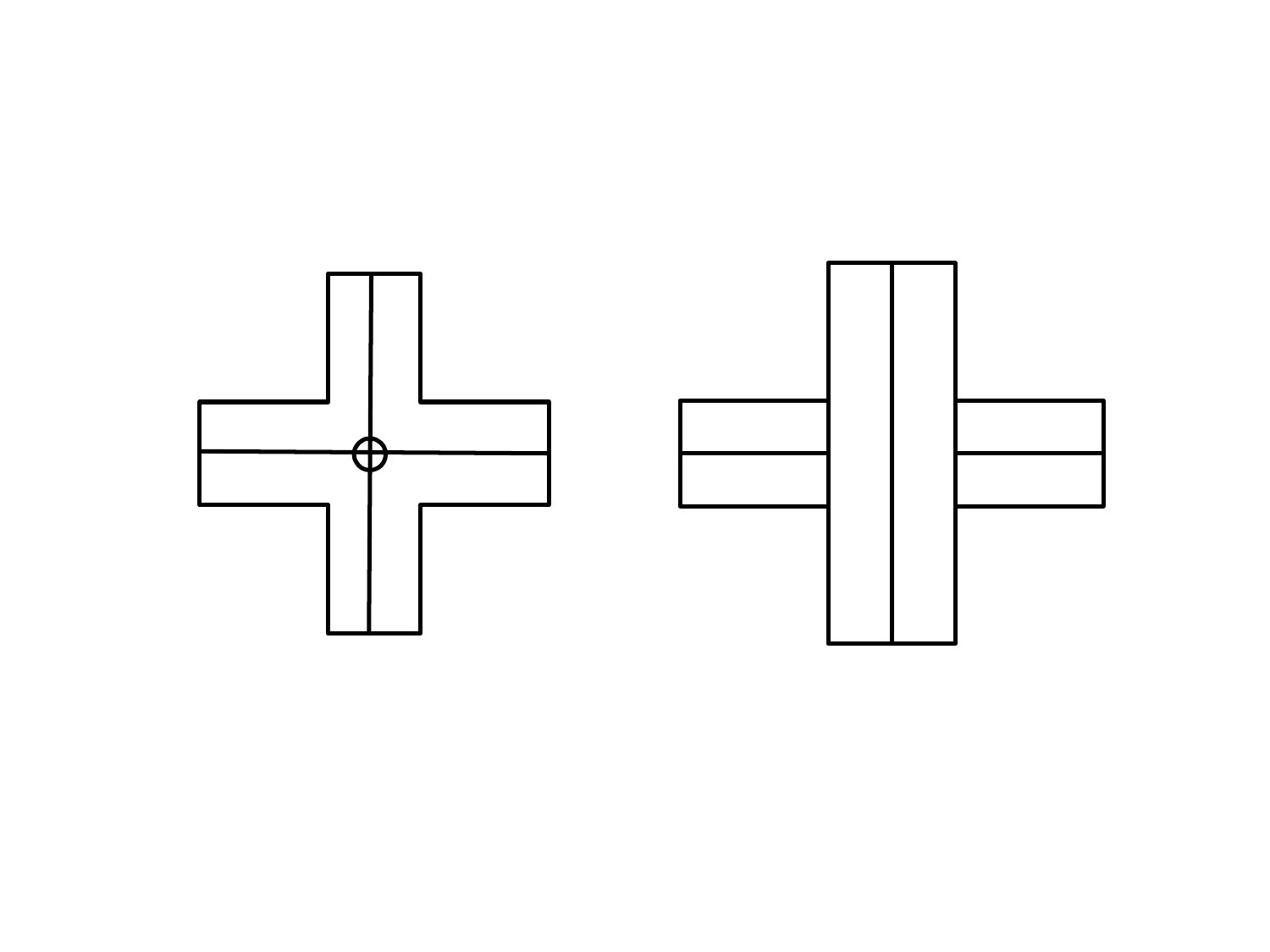}
		\caption{At each virtual crossing in $F\times [0,1]'$, we modify the disk into two disks with an arc crossing each of them.}
		\label{arch}
\end{figure}

The relation between $V_{1}/\cong$ and Kauffman's virtual link is shown by the following theorem, which is proved by a fairly easy verification.

\begin{theorem}\emph{(}\cite{KamaKama, StableEq}\emph{)}
$f$ and $g$ are bijective inverses between the set of virtual link diagrams modulo the virtual Reidemeister moves and $V_{1}/\cong$.
\end{theorem}

Therefore, $V_{1}/\cong$ gives a completely geometric definition for Kauffman's virtual links. It has the particular advantage that it does not explicitly rely upon diagrams or Gauss codes for its definition, which makes it possible to generalize to higher dimensions without first needing to establish a higher-dimensional analog for diagrams or Gauss codes.

This equivalence of $V_{1}/\!\!\cong$ to Kauffman's virtual links makes it geometrically natural to simply define virtual $n$-links as elements of $V_{n}/\cong$.

\begin{definition}
A \emph{virtual $n$-link} is an element of $V_{n}/\cong$. We denote the equivalence class of $(F\times [0,1],L)$ by $[F\times [0,1],L].$
\end{definition}

Elements $[F\times [0,1],L]\in V_{n}/\cong$ such that $L$ has exactly one component will be termed \emph{virtual $n$-knots}. The set of virtual $n$-knots will be denoted by $VK_{n}$. Furthermore, it may easily be seen that for $(F\times [0,1], L), (F'\times [0,1], L')\in V_n$, $(F\times [0,1], L)\cong (F'\times [0,1], L')$ only if $L$ and $L'$ are diffeomorphic.

\begin{definition}
We will denote the set of virtual $n$-links $[F\times [0,1], L]$ with a fixed $L$ by $V(L)$.
\end{definition}

Our definition of virtual links may easily be generalized to virtual tangles. Let $VT_n$ consist of pairs $(F\times [0,1], L)$, for some $(n+1)$-manifold $F$, and $L$ is an $n$-manifold properly embedded in $F\times (0,1)$. Note that this implies that $\partial L$ is embedded in $\partial F\times (0,1)$. Let $(F_{1}\times [0,1], L_{1})\sim(F_{2}\times [0,1], L_{2})$ iff there exists an embedding $f:F_{1}\rightarrow F_{2}$ such that $(f\times id_{[0,1]})(L_{1})=L_2$. We define $\cong$ to be the equivalence relation generated by $\sim$ and by smooth isotopy of $L$, keeping the boundary of $L$ embedded in $\partial F\times (0,1)$.

\begin{definition}
A \emph{virtual $n$-tangle} is an element of $VT_{n}/\cong$.
\end{definition}

\begin{definition}
A \emph{realizable $n$-link} is a virtual $n$-link whose equivalence class in $V_{n}$ includes a pair $(S^{n+1}\times [0,1], L)$.
\end{definition}

Equivalently, a realizable $n$-link is a virtual $n$-link which is virtually equivalent to some pair $(D^{n+1}\times [0,1], L)$. We will use either characterization interchangeably.

It is known that not every element of $V_{1}$ is realizable. For higher $n$, the $L$ in $(F\times [0,1], L)$ may have the diffeomorphism type of a manifold that does not embed into $S^{n+1}\times [0,1]$, in which case this pair cannot be realizable. Some simple examples of this may be found in odd-dimensional real projective spaces; for example, $\mathbb{R}P^5$ cannot be embedded in $\mathbb{R}^7$, \cite{Davis}.

\begin{question}
For a fixed number $n$, are there non-realizable elements $(F\times [0,1], L)$ of $V_{n}$ where $L$ is a manifold that embeds in $S^{n+1}\times [0,1]$?
\end{question}

%

\section{Homotopy Invariants for Virtual $n$-Links}
\label{s-homotopy-inv}
%

In this section we define some invariants of  virtual $n$-links. If these invariants are extensions of classical invariants, i.e. they agree with some classical invariant of $n$-links whenever $(F\times [0,1], L)$ is realizable, then they also can be used to study the relationship between classical $n$-links and virtual $n$-links. We will begin with the following

\begin{remark}\label{abstractknot}
Any element $(F\times [0,1], L)\in V_{n}$ is virtually equivalent to an abstract $n$-link $(F'\times [0,1], L)$ where $F'$ is a tubular neighborhood of  $\pi(L)$ in $F$, and $\pi$ is the canonical projection $\pi:F\times [0,1]\rightarrow F$.
\end{remark}


For a pair $(F\times [0,1], L)\in V_{n}$ or $(F\times [0,1], L)\in VT_n$, let $h(F)$ denote the space $(F\times [0,1])/r$, where $r$ is the relation  $(x, 1) r (x', 1)$ for all $x, x'\in F.$
Note that this is equivalent to gluing a cone over $F$ to $F\times\{ 1 \}$.
We can then use the pair of spaces $(h(F), L)$ to study virtual links, although this pair is not an element of $V_n$.

\begin{theorem}
The homotopy type of $h(F)-L$ is invariant under virtual equivalence.\label{hinv}
\end{theorem}

\bpr It is immediate that isotopy of $L$ in $F\times [0,1]$ will not change the homotopy type of $h(F)-L$.
Let $F'\subset F$. We wish to show that $h(F)-L$ has the same homotopy type as $h(F')-L$. We will define a deformation retraction $H: h(F)-L\to h(F')-L$
defined as follows: $H$ is the identity on $h(F')$ and it is given by a projection
of $(F-F')\times [0,1]$ onto $(F-F')\times \{1\}\cup \p F'\times I$ as depicted in Fig. \ref{vpush}. Hence, $H$ maps $h(F)-L$ to $h(F')-L.$ It is easy to see that this map is homotopic to the identity and, hence, as a deformation retraction, it is a homotopy equivalence between $h(F)-L$ and $h(F')-L.$
\epr

\begin{figure}
		\centering
			\includegraphics[scale=0.5]{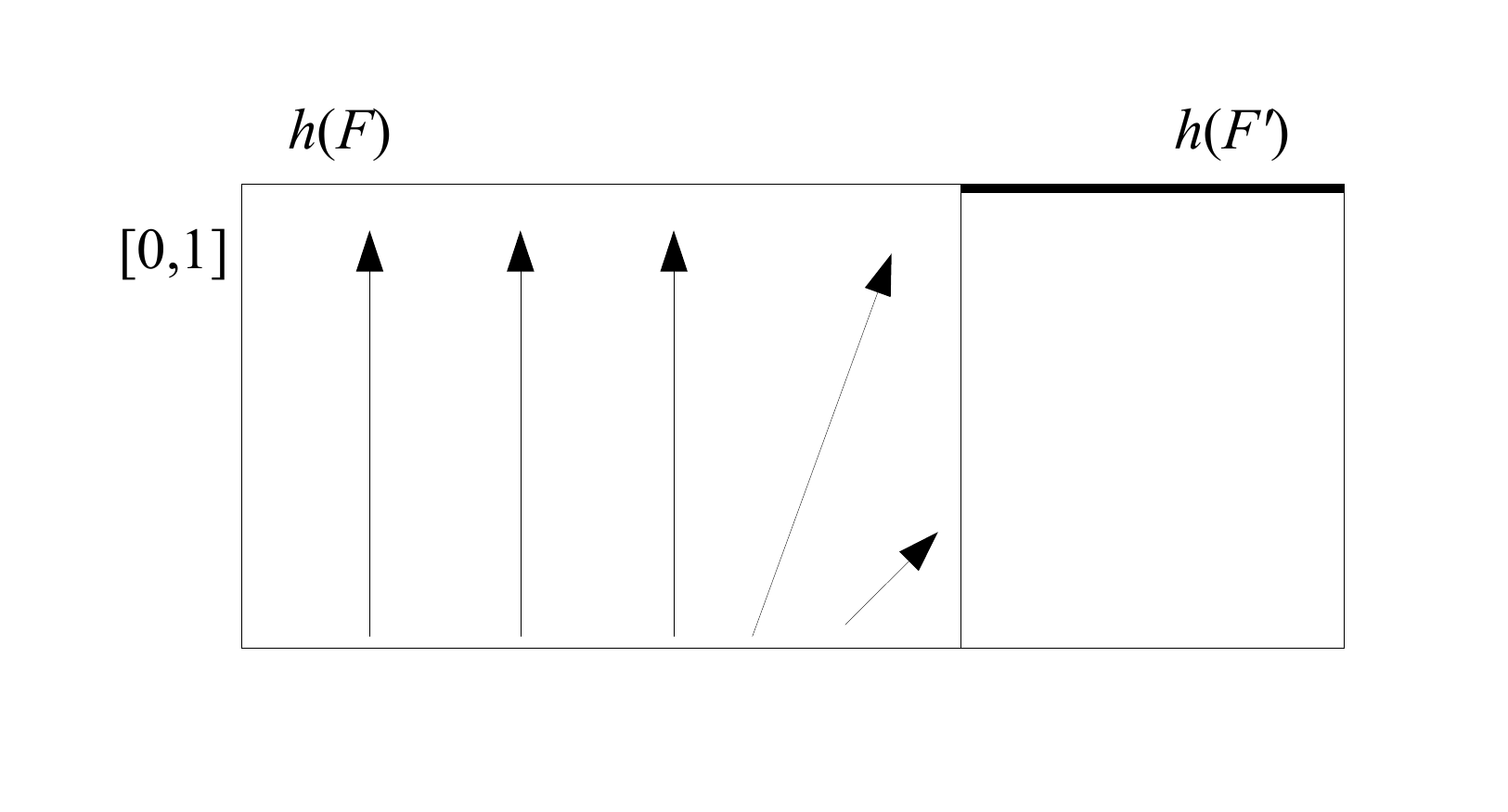}
		\caption{The map $r$ projects points that lie outside a neighborhood of $F'$ straight up, while points in this neighborhood are pushed up and toward the $\partial F'\times[0,1]$.}
		\label{vpush}
\end{figure}

\begin{remark}
For $L\subset D^{n+1}\times [0,1]$, $h(D^{n+1})-L$ is homotopy equivalent to $(D^{n+1}\times [0,1])-L$.
\end{remark}

It follows that any homotopy invariants of links in $D^{n+2}$ extend to invariants of virtual links.

\begin{remark}
We may also consider $F\times [0,1]/r'-L$, where $r'$ is the relation  $(x, 0) r (x', 0)$ for all $x, x'\in F.$ This space is homotopy equivalent to $h(D^{n+1})-L$, for links in $D^{n+2}$, but in general it is not clear whether this will be the case.
\end{remark}

\begin{definition}
(1) We define the link group of a virtual link $[F\times [0,1], L]$ to be $\pi_1(h(F)-L)$.\\
(2) The image of $\pi_1(\p N(L))$ in $\pi_1(h(F)-N(L))$ is called the peripheral subgroup of the virtual link group.\\
(3) For a knot $K\subset F\times [0,1]$, consider a base point $b\in h(F)-K$ which lies in the boundary of $N(K).$ That base point defines the meridian $m\in \pi_1(h(F)-K)$ given by the boundary of a disk in $N(K)$ normal to $K$, passing through $b$. Note that a meridian $m$ is defined up to conjugation in $\pi_1(h(F)-K)$ if a base point is not specified).\\
(4) An oriented closed curve on $l\subset \p N(K)$ passing though $b$ such that the algebraic intersection number between $\pi(l)$ and $\pi(K)$ vanishes in $F$ is called a longitude.
\end{definition}

Note that for virtual $1$-knots a longitude as above is unique and that $m,l$ commute. Furthermore, $m,l$ are uniquely defined up to mutual conjugation by an element of $\pi_1(h(F)-K)$, if the base point is not specified. In other words, $m,l$
define a peripheral structure on $\pi_1(h(F)-K)$.

Observe, however, that higher-dimensional knots may have many longitudes (even if $b$ is specified), which are not conjugate, depending on the topology of the knotted space.




\begin{theorem}
Suppose $(S^{2}\times [0,1], K)\cong (S^{2}\times [0,1], K')$ as elements of $VK_{1}$. Then $K$ and $K'$ are ambient isotopic in $S^{2}\times [0,1]$. \label{ckvk}
\end{theorem}

\bpr The knot group, meridian, and longitude of $K\subset h(F)$ are preserved under virtual equivalences. Since by a theorem of Waldhausen, \cite{Wald}, and a result of Gordon and Luecke, \cite{GL}, these form a complete invariant for knots in $S^{3}$, and since knots in $S^{3}$ are equivalent to knots in $S^{2}\times [0,1]$, the theorem follows.
\epr

Joyce's geometric definition of the fundamental quandle for a codimension-$2$ link, \cite{DJ}, recalled in Section \ref{geoqu}, applies to links in spaces $h(F)$ despite the fact that $h(F)$ is not a manifold.  
In Joyce's construction, we consider paths $\gamma$ from a basepoint $x_0\in F\times [0,1]-L$ to $\partial N(L)$, up to homotopy through such paths. Note that $\partial N(K)\cong K\times S^1$, a circle bundle over $L$.  Given two such paths, $\gamma$ and $\gamma'$, we define $\gamma^{\gamma'}$ to be the path $\gamma\circ\overline{\gamma'}\circ m\circ \gamma'$, where $m$ is the path circling the fiber of $\partial N(L)$ at the terminal point of $\gamma$, as illustrated in Fig. \ref{qop}. The proof that this defines a quandle, and the theorems relating the quandle of $L\subset h(F)$ to the fundamental group of $h(F)-L$ and its peripheral structure when $L$ is an $n$-knot, follow without change from those given in Section \ref{geoqu}.

\begin{definition}
We define the link quandle of a virtual link $[F\times [0,1], L]$ to be the quandle of $L\subset h(F)$, $Q(h(F), L)$.
\end{definition}

\subsection{Computation of the Knot Group and Quandle}

Although we have given a purely geometric definition of the group and the quandle of virtual $n$-links, by using the pair of spaces $(h(F), L)$, it is often convenient to compute these invariants by using link diagrams on $F$. 
We will give an algorithm for doing so, and show that the result does not depend upon performing an isotopy of $L$ or on changing $(F\times [0,1],L)$ by virtual equivalence. The algorithm given here and its independence upon isotopy is discussed in \cite{PR}, however, in that paper, changes to $(F\times [0,1],L)$ by virtual equivalence are not considered.

Let $[F\times [0,1], L]\in V_n/\cong$. We may take a representative $(F\times [0,1], L)$ for this equivalence class such that $L$ projects to a link diagram $d$ on $F\times \{0\}$. Recall that in Section \ref{JR} we constructed the group $G(d)$ and the quandle $Q(d)$ from the faces of diagram $d$ of $L$ and from its double point set.

\begin{theorem}
$G(d)\cong \pi_{1}(h(F) - L)$ and $Q(d)\cong Q(h(F), L)$ for a virtual link $(F\times [0,1], L)$, with $L$ in general position with respect to $\pi$, represented by a diagram $d\subset F.$
\end{theorem}

\begin{proof}
The proof follows from the Van Kampen theorem for link groups and quandles; see \cite{KamaWirt, PR} for the details. We will summarize the proof for the link group; the proof for quandles is almost identical.
Let $(F\times [0,1], L)\in V_n$ with $L$ in general position with respect to $\pi:F\times [0,1]\rightarrow F$.
Let $N$ be an open tubular neighborhood of $D$ in $L$ (where $D$ is the double point set; recall that $B\subset D'$). Then $\pi|_{L-N}$ is an embedding into $F$. 
By an isotopy of $L$ which keeps $\pi(L)$ fixed, we can move $(L-N)$ into $F\times [1/2, 1)$, while moving the interior of $N$ into $F\times (0, 1/2)$. Note that each component of $L-N$ naturally corresponds to a face in the diagram of $L$ on $M$.
Let $h(F, L)_+=h(F)-(F\times [0, 1/2])-(L\cap F\times (1/2, 1]))$. This is a contractible space with the faces of $L$ removed from it.
Then $\pi_1(h(F, L)_+)$ is a free group on the faces in the diagram of $L$.
To complete the computation of the link group, we follow the method of \cite{KamaWirt}. For each component of the set of pure double points $D_0$, we must glue on a neighborhood of $N$. Such a neighborhood is shown in \cite{KamaWirt} to have a trivialization which is determined by the four sheets meeting at this double point set. Because $F$ and $L$ are orientable, these sheets can be co-oriented in the diagram of $L$ on $F$. Using the Van Kampen theorem inductively on each component, it is shown in \cite{KamaWirt} that each component of the pure double point set in the diagram contributes a pair of relations between the generators of $\pi_1(h(F, L)_+)$ with the form $w=y, x=yzy^{-1}$, where $w$ and $y$ are the sheets of the overcrossing face, and the normal vector to $y$ pointing toward $x$. Furthermore, no additional relations are introduced from the points in the double point set that are not pure double points.
Thus, $\pi_1(h(F)-L)$ has a presentation whose generators are the faces of a diagram of $L$ on $F$, and with one relation for each meeting of three faces in a double point set as described in the definition of the group $G(\pi(L))$.
\end{proof}

\begin{remark}\label{Hcyclic}
Note that for a link diagram $d$ of $L\subset F\times [0,1]$, all generators in $G(d)$ corresponding to faces of the same connected component of $L$ are conjugate. From that, it is easy to see that the abelianization of $G(d)$ is $\Z^n$, where $n$ is the number of connected components of $L.$ Consequently, $H_1(F\times [0,1]-L)=\Z^n.$
\end{remark}


The computations of $\pi_1(h(F)-L)$ and of $Q(h(F),L)$ via diagrams suggest that many invariants defined for classical $n$-links by referencing their sheets and double point sets and proving invariance under Roseman moves can be extended to invariants of virtual $n$-links. As an example, we can extend the biquandle invariant of $2$-links to virtual $2$-links. The biquandle invariant for $2$-links is defined in \cite{Carrell}, and shown to be an invariant. Given a $2$-link $L$ in $F\times [0,1]$, Carrell constructed a biquandle $B(F\times [0,1], L)$ which is generated by the sheets (rather than the faces, as in the case of the quandle or group) of $L$'s diagram. For each double point curve in the diagram separating sheets $a$ and $d$, as shown in Fig. \ref{biquand}, he introduced the following relations $c=b_{a}, d=a^{b}$. The resulting biquandle was shown to be well-defined and independent of the Roseman moves in \cite{Carrell}. Given a pair $(F\times [0,1], L)\in V_{2}$, we can define $B(F\times [0,1], L)$ in an identical manner.

\begin{figure}[htbp]
\begin{centering}
\includegraphics[scale=0.3]{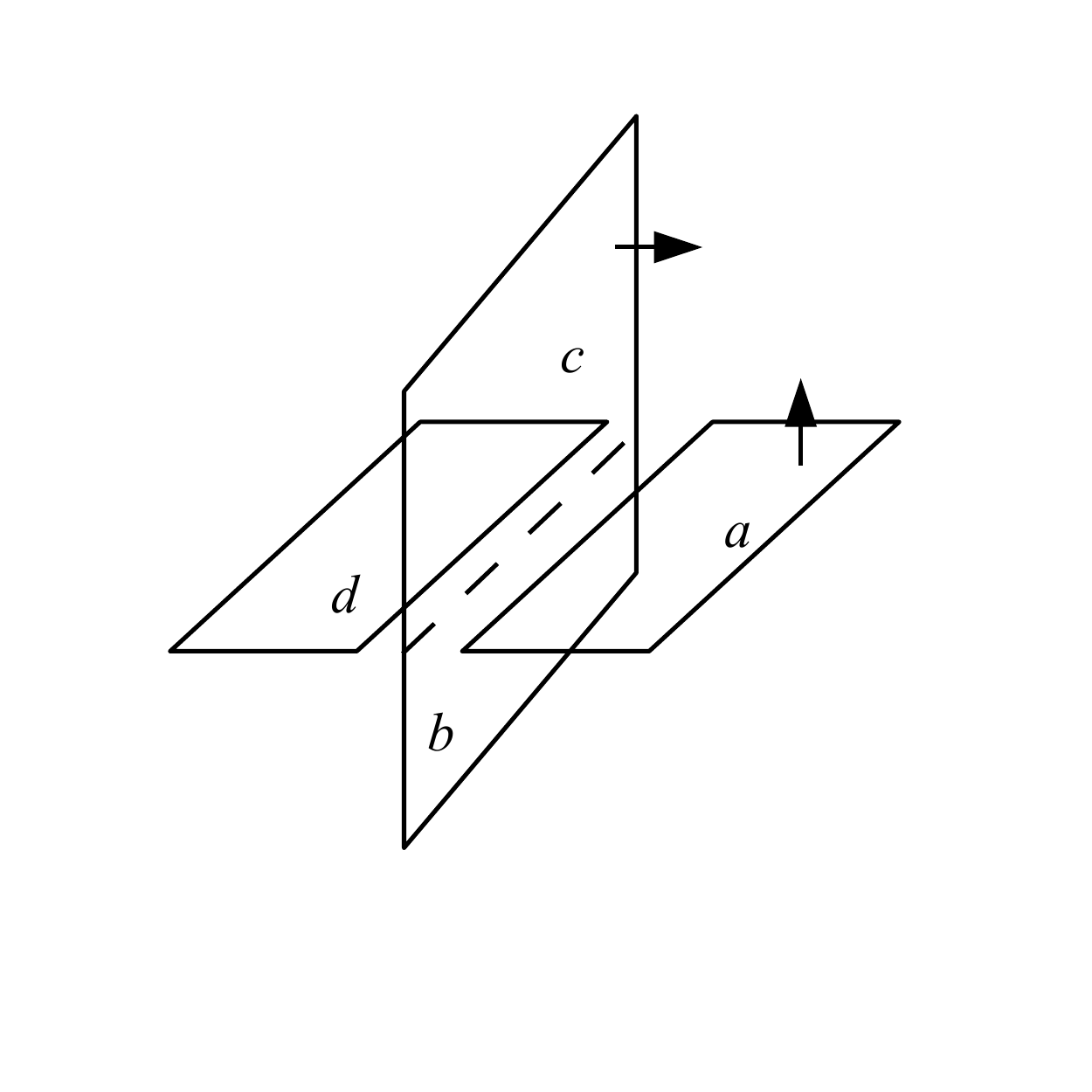}
\caption{A double point curve as shown yields relations $c=b_{a}, d=a^{b}$ on the biquandle $B(F\times [0,1], L)$}
\label{biquand}
\end{centering}
\end{figure}

\begin{theorem}
The biquandle $B(F\times [0,1], L)$ is invariant under virtual equivalence.
\end{theorem}
\bpr The proof requires two steps. First, the biquandle must be shown to be invariant under Roseman moves on the diagram of $L$, which involves checking the effect of each Roseman move on the biquandle. This is verified in \cite{Carrell}. Then we must show that it does not change under virtual equivalences that change $M$. This, however, is immediate from the definition, which depends only on the sheets, faces, and double point curves. \epr

Biquandles thus define invariants of both virtual surface links and classical surface links.

\begin{question}
Is it possible to give a purely geometric definition of $B(F\times [0,1], L)$ which does not use link diagrams?
\end{question}

\begin{question}
For $n\geq 3$, it is possible to define a biquandle using sheets and crossing sets in a diagram to define generators and relations in a similar manner. Is this biquandle well-defined and invariant for $n$-links?
\end{question}

The group, quandle, and biquandle invariants can be difficult to apply in practice due to the difficulty involved in determining whether two groups, quandles, or biquandles are isomorphic. For this reason in practice it is desirable to use other invariants which may be more easily compared. One such invariant, defined for any $n$-link and which we can extend to virtual $n$-links, is the notion of a virtual link $(F\times [0,1], L)$ being colorable by a finite quandle $R$, meaning that there is a surjective homomorphism $Q(h(F), L)\rightarrow R$. Quandle colorings can be used to construct other invariants, such as the quandle cocycle invariants, \cite{PR}, (see \cite{CKS} for additional discussion of cocycle invariants in the case of $2$-links). These coloring and cocycle invariants have the advantage of being more easily compared.

\begin{theorem}
The quandle cocycle invariants defined for $n$-links in \cite{PR} are invariants for virtual $n$-links.
\end{theorem}
\bpr The definitions of quandle cocycle invariants are constructed by considering the sheets and double points of a link diagram and showing the result is independent of Roseman moves. The ambient space for the diagram is not relevant to the definition. Therefore, the quandle cocycles are invariant under changes to the diagram by Roseman moves and invariant under virtual equivalence. \epr



\begin{theorem}
Suppose $K\subset F\times [0,1]$ is a ribbon $n$-link, $n\geq 2$, with ribbon solid $R\subset F\times [0,1]$. Then $K$ is virtually equivalent to a ribbon link in $D^{n+1}\times [0,1]$ bounding a ribbon disk with the same ribbon presentation.
\end{theorem}


\bpr Let $\Gamma$ be a core of the ribbon solid $R$ -- that is a $1$-dimensional CW-complex inside $R$ to which $R$ retracts to.
We can assume that $\Gamma$ is in a general position with respect to the canonical projection $\pi:F\times [0,1]\to F$. Then $\pi(\Gamma)$ is a finite $1$-dim CW-complex and, by ``shrinking'' $R$ towards $\Gamma$ if necessary, we can assume that $\pi(\Gamma)$ is a deformation retract of $\pi(R).$

Furthermore, for a sufficiently small closed neighborhood $N$ of $\pi(R)$ in $F$, that neighborhood deformation retracts to $\pi(\Gamma)$ as well.
Note that $(F\times [0,1],K)$ is equivalent to $(N\times [0,1],K).$ Since $F$ is orientable, $N$ is an orientable $n+1$ dimensional handle body. That handlebody embeds into $D^{n+1}$ for $n\geq 2$. (Note that $n\geq 2$ may be necessary, since $\pi(\Gamma)$ does not have to be a planar graph.
\epr

For virtual knots $K\subset F\times [0,1]$, $H_1(h(F)-K)=\Z$ by Remark \ref{Hcyclic}. It follows that the action of the deck transformations on the homology groups of the universal abelian cover $\widetilde{h(F)-K}$ will make them into $\Z [t, t^{-1}]$ modules. As such, we can construct Alexander-type invariants for virtual $n$-knots.

Theorem \ref{ckvk} is an important result for the application of virtual $1$-links to the study classical links, since it shows that two classical knots in $S^2\times [0,1]$ which are virtually equivalent must be classically isotopic as well. It is thus natural to ask whether an analogous result holds for higher dimensions. As Theorem \ref{hinv} shows, many homotopy invariants of $n$-links in $(n+2)$-balls will extend to invariants of virtual $n$-links. However, for $n\geq 2$, the homeomorphism type of a knot complement does not determine the knot up to isotopy. Examples of inequivalent knotted spheres with homeomorphic complements are given in \cite{Gor}. See \cite[Section 3.1.3]{CKS} for further discussion of this problem.

This means that the homotopy type of $h(F)-K$, while it is a powerful enough invariant to allow many homotopy invariants of classical links to be extended to virtual links, will not allow a complete invariant of classical $n$-links to be extended to virtual $n$-links, even taking the peripheral structure into account. In fact, in the case of tori in a $4$-ball, we can give an example of two virtually equivalent knots which are not isotopic in the $4$-ball. For this example, we must define the \emph{spun torus} for a knot in $\R^{3}$. This is a torus in $\R^{4}$ (or equivalently in $D^{4}$ or $S^{4}$) constructed as follows. Let $\R^{3}_{+}=\{ (x, y, z)\in \R^{3} | x \geq 0\}$ with $K\subset \R^3$ a knot, and isotope $K$ so that $K\subset \R^{3}_{+}$. We can consider an open book decomposition of $\R^{4}$ given by $\R^{3}\times S^{1}/(0, y, z, t)\sim (0, y, z, t')$. Then $K\times S^{1}\subset \R^{3}\times S^{1}/(0, y, z, t)\sim (0, y, z, t')$ is a torus embedded in $\R^{4}$. We will denote this torus as $Spun(K)$.

A similar construction yields the \emph{$n$-turned spun torus} corresponding to $K$. This is defined in the same way except that instead of taking $K\times S^{1}$ directly, we instead rotate $K$ by $n$ complete turns as we circle the open book decomposition, \cite{CKS}. The result of this will be denoted $Spt_{n}(K)$. Note that $Spun(K)=Spt_{0}(K)$. In \cite{Boy}, Boyle defines a construction which shows that for $T$ the trefoil knot, $Spt_{1}(T)$ is not a ribbon knot, using invariants derived from the second homology of the complement of the torus in $S^{4}$ (cf. also the discussion in \cite{CKS}). On the other hand, $Spun(T)$ is ribbon (and in fact, the spun torus for any classical knot is ribbon, \cite{CKS, SS}). Therefore, in particular $Spt_{1}(T)$ and $Spun(T)$ are not isotopic in $S^{4}$ and thus cannot be isotopic in $D^{3}\times [0,1]$. Recall that knots in $S^{4}$, $\R^{4}$, and $D^{4}$ are isotopic in one of these spaces iff they are isotopic in all of them.


\begin{theorem}
Let $K$ be a $1$-knot in $D^{3}$. Then for any $m, n$, $Spt_m(K)$ and $Spt_{n}(K)$, considered as knots in $D^{3}\times [0,1]$, are virtually equivalent.
\end{theorem}


\bpr
By construction, both $Spt_m(K)$ and $Spt_n(K)$ are contained inside inside a $D^2\times S^1\times [0,1]$ subspace. It follows that $Spt_m(K)$ is virtually equivalent to $(D^2\times S^1\times [0,1], Spt_m(K))$. Likewise, $Spt_n(K)\cong (D^2\times S^1\times [0,1], Spt_n(K))$. It is then straightforward to construct a map $f:D^2\times S^1\rightarrow D^2\times S^1$ such that $f\times id_{[0,1]}$ carries $Spt_m(K)$ to $Spt_n(K)$, by letting $f$ twist the solid torus $D^2\times S^1$ by $n-m$ turns. Then $f$ induces a virtual equivalence between $(D^2\times S^1\times [0,1], Spt_m(K))$ and $(D^2\times S^1\times [0,1], Spt_n(K))$.\epr

It is in fact reasonable to expect that the $Spt_n$ construction will provide many examples of classical $2$-links which are classically distinct but virtually equivalent. Such an example will be produced any time that $Spt_m(K)$ and $Spt_{n}(K)$ are not classically isotopic. However, classifying $Spt_n(K)$ as a classical surface knot for general $K$ and $n$ is an open problem.

\begin{corollary}
There exists a ribbon $2$-knot (specifically the knotted torus $Spun(T)$) in $D^{3}\times [0,1]$ which is virtually equivalent to a non-ribbon knot, $Spt_1(T)$.
\end{corollary}

On the other hand, the above examples only involve knotted tori, and their construction makes use of embedding a solid torus into $S^4$ with a ``twist'' that is not detected by virtual knots. It is, therefore, still an open question whether there are examples of knotted spheres in $D^{3}\times [0,1]$ which are virtually equivalent but not classically isotopic.

\begin{conjecture}
For $n\geq 2$ there are knotted $n$-spheres in $D^{n+1}\times [0,1]$ which are virtually equivalent but not classically isotopic.
\end{conjecture}

As stated previously, the homeomorphism type of a knot complement does not necessarily determine the knot for $2$-knots. This is true even for $2$-spheres: there is an infinite family of pairs of knotted $2$-spheres such that the spheres in each pair are non-isotopic but have homeomorphic complements, \cite{Gor}. Such knots may be good candidates for proving this conjecture.

\subsection{A Method for Constructing Virtual Links}\label{spinning}
Let $N$ be a fixed $n$-manifold. Then we can define a map $\Phi_N:V_m\rightarrow V_{m+n}$. Given a pair $(F\times [0,1], L)\in V_m$, $L$ is embedded in $F\times [0,1]$ by some map $f$. Then we define $\Phi_N(F\times [0,1], L)$ to be the pair $(F\times N\times [0,1], L\times N)$, where $L\times N$ is embedded into $F\times N\times [0,1]$ by the map $f\times id_N\times id_{[0,1]}$. It is easy to see that this map descends to a map $V_m/\cong\rightarrow V_{m+n}$.

\begin{conjecture}\label{VK_m+n-conj}
For any $N$, the map $\Phi_N:V_m\rightarrow V_{m+n}$ has the property that $\Phi(F\times [0,1], L)\cong \Phi(F'\times [0,1], L')$ only if $(F\times [0,1], L)$ and $(F'\times [0,1], L')$ are virtually equivalent up to mirror images and orientation reversal of $L$.
\end{conjecture}


\subsection{Nontrivial Virtual $2$-Knots With Infinite Cyclic knot Group}

We will now give an example of a virtual $2$-knot which is nontrivial but has the same group and quandle as the trivial knot. Such examples are known for $1$-knots, and our constructions will be based on one of them. Let $KI$ denote the Kishino knot, shown in Fig. \ref{KI}. This is a virtual knot which is distinguished from the unknot by its biquandle, \cite{Kauff2, BF}. On the other hand, as a welded knot it is trivial, as it may be unknotted using the forbidden move. Since the knot quandle and group are invariants for welded knots, this shows that the quandle and group of the Kishino knot are isomorphic to the quandle and group of the unknot.
We now wish to construct different elements of $V_2/\cong$ with the same quandle, group, and biquandle as those of the Kishino knot.
We will use a generalization of the spinning construction for classical knots, \cite{CKS}. Let $(F\times [0,1], K)$ be a virtual $1$-knot. Then we will define $S_T(F\times [0,1], K)$ to be the pair $(F\times S^1\times [0,1], K\times S^1)$. $S_T(F\times [0,1], K)$ is clearly a virtually knotted torus. The map $S_T$ is a special case of the construction defined in Section \ref{spinning}.

\begin{theorem}
$S_T$ preserves the group, quandle, and biquandle, in the sense that the knot group, quandle, and biquandle of $S_T(K)$ will be isomorphic to the group, quandle, and biquandle of $K$.
\end{theorem}

\bpr All three of these invariants can be computed from the faces and sheets together with the double point sets of a diagram of a knot. The construction yields the same presentations for a $1$-knot $K$ as for $S_T(K)$, since this has the same faces and sheets and each crossing in the diagram of $K$ yields a single double point curve at which the same faces and sheets meet with the same resulting relations. \epr

\begin{corollary}
$S_T(KI)$ is a virtual $2$-knot of a torus with the knot group and knot quandle of the unknot, but nontrivial biquandle.\label{2Kish}
\end{corollary}

We note that either $S_T(KI)$ provides an example of a nontrivial classical knotting of a torus with infinite cyclic knot group, or else it is a virtually knotted torus which is not realizable.

\begin{conjecture}
$S_T(KI)$ is not realizable.
\end{conjecture}

\section{Gauss Codes for Higher Dimensions}
\label{s-gauss-higher-dim}
%

For a classical link, a Gauss code consists of a space $X$ of disjoint copies of $S^{1}$, together with ordered triples $(x,y,z)$, where $x,y$ are a pair of points in $X$ and $z$ is a $+$ or $-$ sign. Any link diagram with $n$ components gives rise to a Gauss code by letting $X$ be the space consisting of $n$ copies of the circle, and introducing a triple $(x,y,z)$ for each crossing in the diagram, where $x$ is the point of the overcrossing, $y$ is the point of the undercrossing, and $z$ is the sign of the crossing. These may be shown on a diagram by drawing an arrow from $x$ to $y$ and labeling the head of the arrow with the sign of the crossing. A Reidemeister move on the diagram changes the Gauss code in a local way, by introducing an isolated triple (1st Reidemeister move), introducing two parallel triples of opposite sign (2nd Reidemeister move), or rearranging three triples (3rd Reidemeister move).

By allowing arbitrary Gauss diagrams modulo arbitrary changes of the form determined by the Reidemeister moves, we obtain the virtual links of Kauffman, \cite{Kauff}.

\begin{figure}
		\centering
			\includegraphics[scale=0.3]{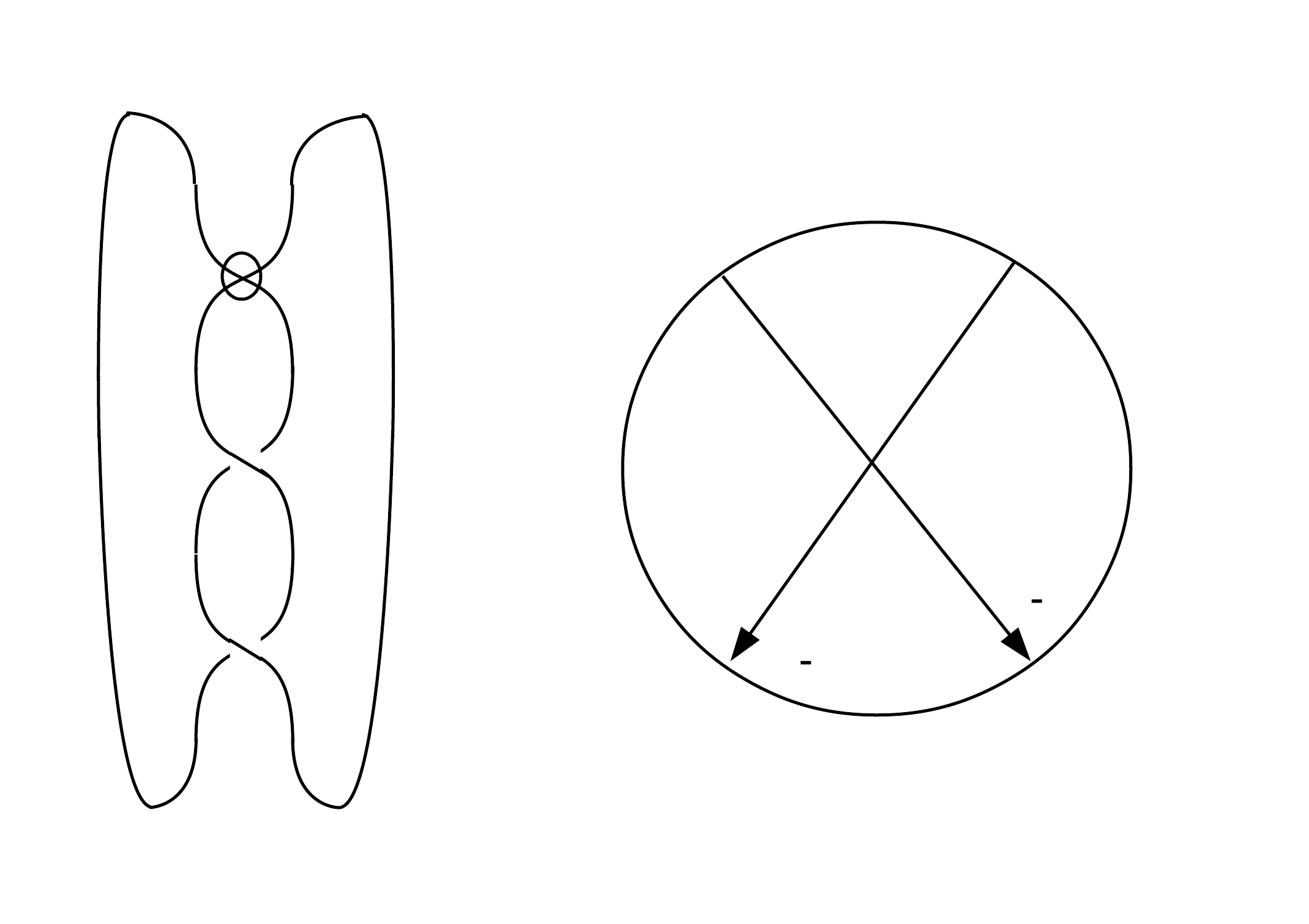}
		\caption{A virtual knot diagram and its Gauss code.}
		\label{gaussex}
\end{figure}

We will now generalize this to $2$-links. Let $L$ be a $2$-link of $X$ in $F\times [0,1]$, in general position with respect to the canonical projection $\pi:F\times [0,1]\rightarrow F$. Then by labeling $\pi(L)$, we can produce a broken surface diagram on $F$. Corresponding to each labeled double point curve in the broken surface diagram, there will be two curves on $X$, and we may pair up these curves, marking which is the over-crossing and which is the under-crossing. We must also mark in which direction the normal vectors of the surfaces point, as indicated in Fig. \ref{gausssurface}. There will also be cusp points and triple points. The triple points will lift to places on $X$ where double point curves cross one another, and if two double point curves cross there must be a third double point curve which shares that triple point. A surface $X$ together with curves marked in this manner gives us an object analogous to a Gauss code for a classical link. Changing the broken surface diagram by a Roseman move will cause a local change to the marked curves on $X$.

Locally, any double point curve may be indicated by noting which sheets meet at the double point curve, which is the over-crossing and which is the under-crossing, and marking in which direction the normal vectors to each sheet are pointing.

\begin{figure}
		\centering
			\includegraphics[scale=0.5]{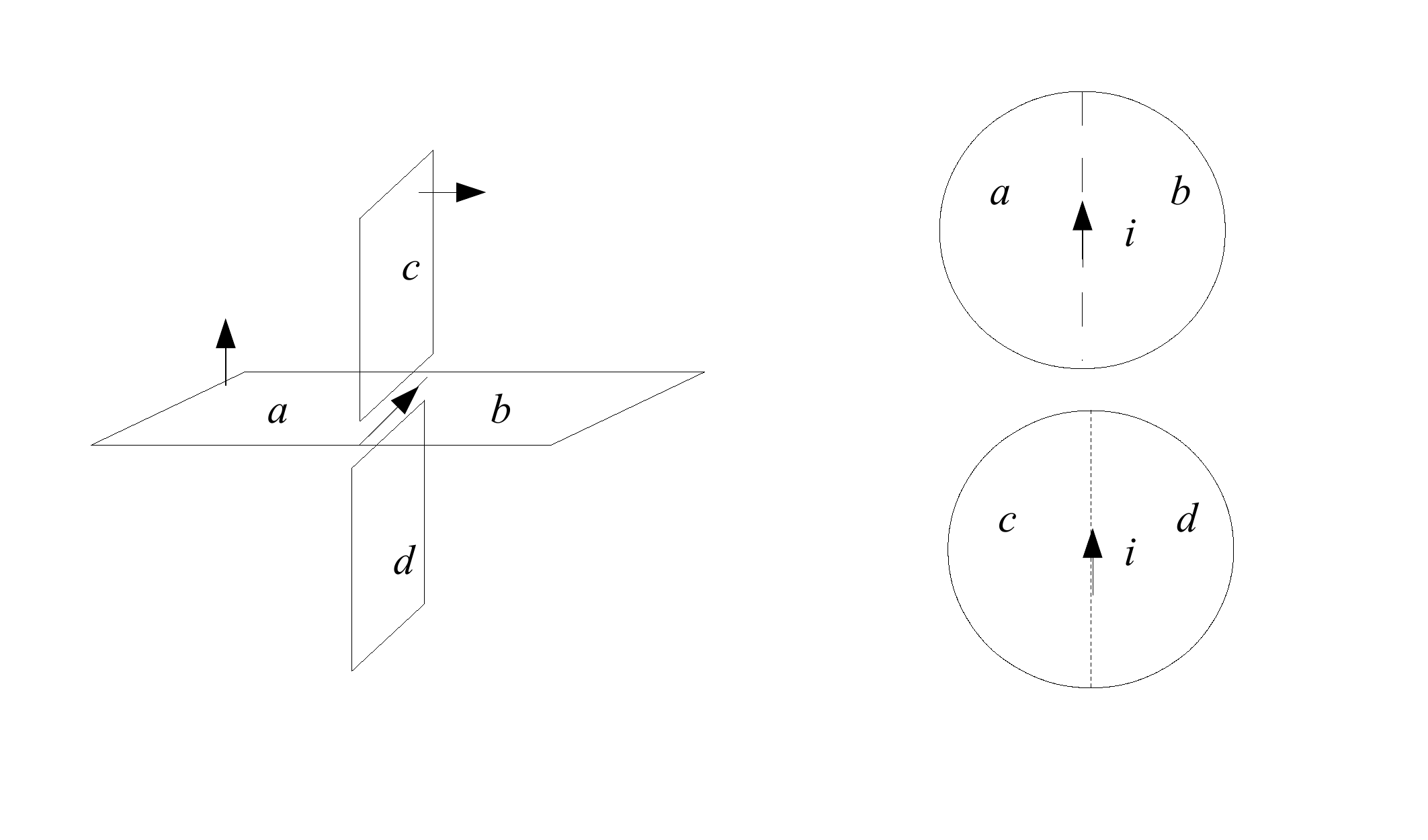}
		\caption{On the surface $X$ we mark the curves where the double point curve lifts. One curve comes from the over-crossing part and one from the under, and they are labeled as coming from the same double point curve with the letter \emph{i}. The overcrossing is marked with the larger dashed curve. Arrows are placed on each curve to determine which direction to pair them together, with the convention that the triple consisting of the normal vector to the overcrossing sheets, normal vector to the undercrossing sheets, and the arrow on the curves, should form a positive basis.}
		\label{gausssurface}
\end{figure}

Suppose now that instead of starting with a broken surface diagram of a knotted surface in four dimensions and lifting the double point curves, we instead consider simply a surface $X$ together with a collection of curves which \emph{locally} could be the lift of the double point curves of a broken surface diagram of a knotted surface. That is, we consider $X$ with a collection of curves which are paired together, with one curve in each pair marked as an \emph{overcrossing}, with the rule that curves may only terminate in cusps and that if two double point curves cross, there is a third curve which shares the triple point. The surface $X$ together with these marked curves will be termed a \emph{(virtual surface) Gauss code}. Although \emph{locally} such a diagram may be obtained as the lift of a neighborhood of a broken surface diagram on a disk, it may not be possible to obtain it in this manner globally.

\begin{definition}
A virtual surface Gauss code is a closed surface $X$ with a set $D$ of marked curves obeying the following conditions.

1. Each curve is closed or terminates in a cusp.

2. Each curve in $D$ is paired with a unique second curve in $D$, and one curve in the pair is marked as \emph{over}. Arrows are placed on each curve in the pair. The normal vectors of the two surfaces are also marked as in Fig. \ref{gausssurface}.

3. Two curves that terminate in a single cusp are paired together, with arrows both pointing toward or both pointing away from the cusp.

4. If two curves cross, then the curves they are paired with cross as well (that is, triple points appear three times on $X$).
\end{definition}

This may be succinctly summarized by stating that $X$ is marked with curves in the set $D$ such that \emph{locally} the set $D$ is the lift of some neighborhood of a broken surface diagram on a disk. Alternatively, $X$ is the surface Gauss code for a broken surface diagram on an arbitrary $3$-manifold.

For broken surface diagrams, there exists a set of seven moves, the Roseman moves, which suffice to generate any isotopy of the surface. Each move creates a localized change to the Gauss code. Our \emph{virtual Gauss code surface link} may thus be taken modulo such moves, where we allow such moves to be performed regardless of whether they would be ``realizable'' in the broken surface diagram. Such changes to a Gauss code will be termed \emph{Gauss-Roseman moves}.

Note that these Gauss-Roseman moves are in fact all changes to the Gauss code which take a neighborhood that is locally the lift of some broken surface diagram for a knotted surface and change that neighborhood by a Roseman move.

%

\section{Broken Surface Diagrams for Virtual Surface Links}
\label{s-broken-surf}
%

At this point we will restrict our consideration to $V_{2}$.
Given a $2$-dimensional surface embedded in $F^{3}\times [0,1]$, it is possible to develop a diagrammatic expression for this embedding analogous to a classical link diagram by labeling a projection of the surface onto $F$, in the manner described by Roseman, \cite{Roseman1, CKS}. Such diagrams are called \emph{broken surface diagrams}, which allow the study of surfaces in four dimensions to be studied via combinatorial methods. We will review the basics of such diagrams for classical $2$-links, and then discuss how such diagrams may be generalized to incorporate virtual $2$-links.

\begin{remark}
J. Schneider has also considered generalizations of broken surface diagrams allowing virtual crossings, \cite{Schneider}. His use of virtual broken surface diagrams is as a combinatorial generalization of broken surface diagrams. This reference also considers an equivalence relation consisting only of virtualized versions of the Roseman moves, and does not appear to allow the \emph{graph changes} which we also permit. Therefore, this appears to be a distinct approach to generalizing $2$-links to virtual $2$-links. There does not appear to be an a priori reason to expect these two approaches to be equivalent.
\end{remark}

In order to form a broken surface diagram for a $2$-link $L$ in $F^{3}\times [0,1]$, consider the canonical projection coming from the Cartesian product $\pi:F^{3}\times [0,1]\rightarrow F^{3}$. By performing an arbitrarily small isotopy, we may place $L$ in general position with respect to $\pi$.

\begin{theorem}\cite{Roseman1}
Given a $2$-link $L$ in general position in $F^3\times [0,1]$, the projection $\pi(L)$ will be embedded except at finitely many umbrella points, finitely many triple points, and double point curves, by. The double point curves are either closed or terminate in umbrella points.
\end{theorem}

\begin{definition}
A \emph{broken surface diagram} on $F^{3}$ is a surface in $F^{3}$ which is embedded except at finitely many umbrella points, finitely many triple points, and a set of double point curves, where each double point curve is marked with over/under information.
\end{definition}

These broken surface diagrams are analogous to classical link diagrams on surfaces. They are also the same as the diagrams defined by Roseman, \cite{Roseman2}, in dimension $4$.

The usefulness of broken surface diagrams is not only in allowing us to express surfaces in four dimensions through diagrams in three dimensions, but also in the existence of a complete set of moves on these diagrams called the \emph{Roseman moves}, which are analogous to the Reidemeister moves for classical link diagrams. These moves are shown in Fig. \ref{Rose}. Their completeness is proved in \cite{Roseman}, with a more general proof applicable to any dimension found in \cite{Roseman2}.

\begin{figure}
		\centering
			\includegraphics[scale=0.5]{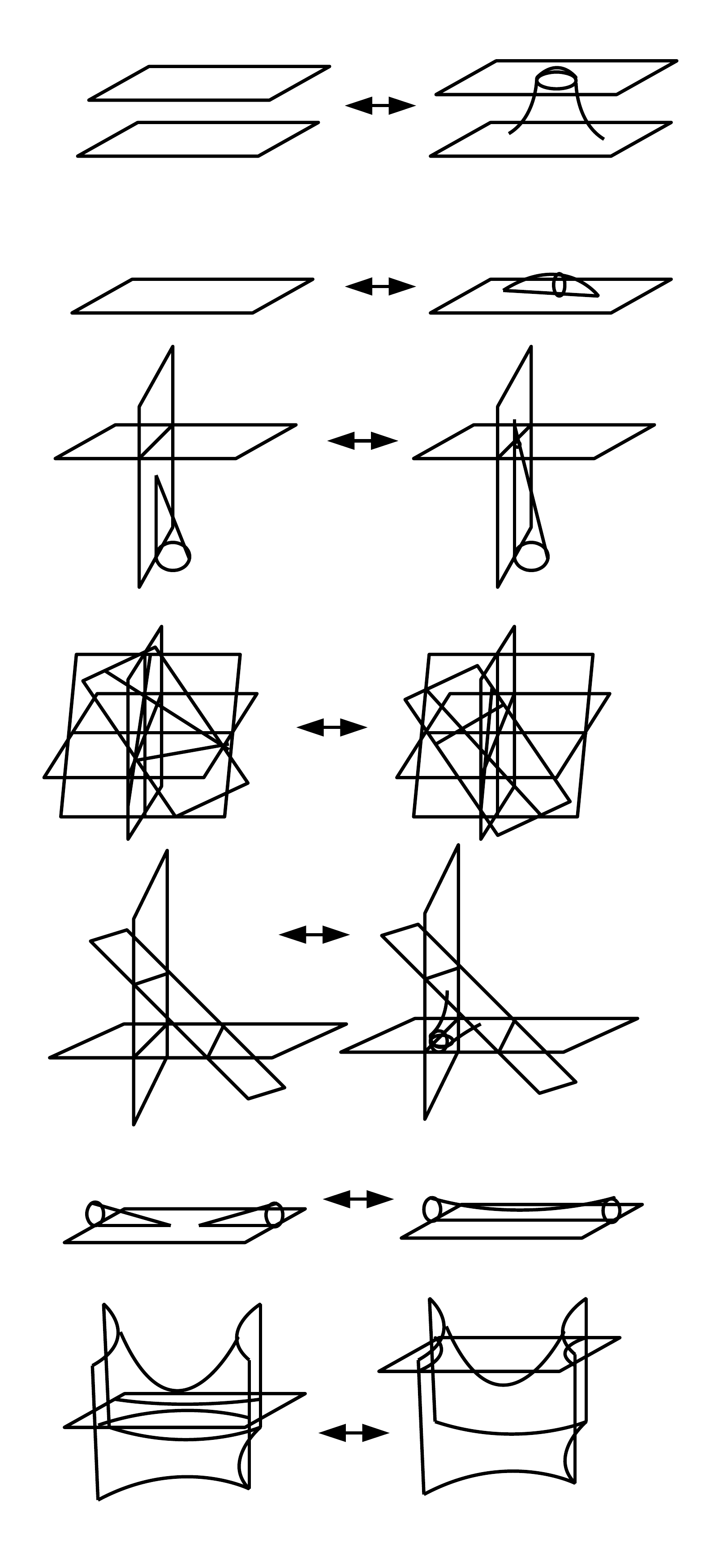}
		\caption{The seven Roseman moves for broken surface diagrams. Over/under and virtual markings on the double point curves must be consistent before and after the move.}
		\label{Rose}
\end{figure}

Now, for $V_{1}$, it is possible to extend classical link diagrams on $S^{2}$ to allow us to describe arbitrary virtual $1$-links. Given a pair $(F^{2}\times [0,1], L)$, we may represent $L$ by a link diagram on $F$. For each crossing in $\pi(L)$, we put a crossing of the same type in a diagram on $S^2$. Then we join together the ends of the segments around each crossing so that we get the same set of arcs as we have in $\pi(L)$, allowing the curves joining the crossings to pass through one another only in virtual crossings. This algorithm involves many choices: we must choose the position for each of the crossings in the diagram, and we can choose different arcs to connect them. However, it is shown in \cite{KamaKama} that the results will always be equivalent under virtual Reidemeister moves. In brief, the proof requires two observations: first, given two choices for locating the crossings on $S^2$, we may isotope these crossings so that their locations agree. This is easy to do since each crossing can be thought of as lying inside a small disk on $S^2$, and we need only slide and rotate the disks so that they agree. On the other hand, if the crossings are located in the same places, then any set of arcs joining the crossings are related by some sequence of virtual replacements. It is furthermore shown that two virtual link diagrams which represent equivalent elements of $V_1$ are related by a finite sequence of local Reidemeister moves together with allowing a semi-arc with only virtual crossings to be replaced by another semi-arc with the same endpoints and only virtual crossings.

We may generalize broken surface diagrams on $S^{3}$ by allowing for double point curves which are marked as virtual instead of being marked with over/undercrossing information. Such diagrams are \emph{virtual broken surface diagrams}, which we may abbreviate for convenience as \emph{VBSDs}. We consider these diagrams modulo the following moves. Given a broken surface diagram $K$, let $\Sigma$ be a subsurface (with boundary) of $K$ such that all the double point curves on $\Sigma$ are virtual crossings or meet $\Sigma$ transversely. Then we may replace $\Sigma$ by another surface $\Sigma'$ with the same genus and boundary as $\Sigma$ and replace $K$ by $K'=(K-\Sigma)\cup \Sigma'$, where all double point curves which meet $F'$ are marked as virtual. Such an adjustment will be called a \emph{virtual replacement}. In addition we will allow a local neighborhood of the diagram involving only classical crossings (that is, crossings marked with over/under information) to be changed by a classical Roseman move.

Instead of considering virtual replacements, one may instead wish to find a local set of moves on virtual crossings which results in the same set of equivalences between diagrams. To do so, we will define the \emph{virtual Roseman moves} to be any local modification of the diagram in which we perform a classical Roseman move, except that one of the disks involved in the classical Roseman move has only virtual crossings with the other disks, both before and after the move is performed.

\begin{theorem}
Two virtual broken surface diagrams are equivalent modulo Roseman moves and virtual replacements iff they are equivalent modulo the virtual and classical Roseman moves.
\end{theorem}
\bpr It is immediate that any two diagrams which are equivalent under the virtual and classical Roseman moves must be equivalent under the classical Roseman moves together with virtual replacements. It remains to show that any virtual replacement can be accomplished using the virtual Roseman moves. The two diagrams differing by a virtual replacement are related by a homotopy which slides the disk that undergoes virtual replacement from its old position to its new position. This homotopy will leave the immersion of the broken surface generic except at a finite collection of times, by the same argument which shows that a generic isotopy of a surface in $\mathbb{R}^4$, projected into three dimensions, leaves a generic immersion in $\mathbb{R}^3$ except at a finite collection of times, \cite{Roseman, Roseman2}. By Roseman's argument for the completeness of the Roseman moves, we may assume that such non-generic immersions are of the sort encountered in the Roseman moves. The only difference will be that in our case, one of the disks involved will have only virtual crossings instead of marked crossings. \epr

Now we wish to establish a relationship between our $V_2/\cong$ and these virtual broken surface diagrams. Let us consider the approach taken for $1$-links by Kamada and Kamada \cite{KamaKama} as summarized above. The first step is to show that there is a way to represent virtual $2$-links by using VBSDs, after which we may consider whether this map respects virtual equivalence. Therefore, let $(F\times [0,1], K)$ be a virtual $2$-link. Consider the broken surface diagram $d$ on $F$. This will have a collection of double point curves which either form closed circles or line segments that terminate in umbrella points, \cite{Roseman, Roseman2, CKS}. We wish to form a VBSD on $S^3$ with the same collection of double point curves and umbrella points. Placing a double point curve in $S^3$ amount to placing a framed link and some framed curves with ends; the framing determines how the sheets and faces intersect at the double point curve. To see this it suffices to consider that, for a point on a double point curve that is not an umbrella point: in a ball neighborhood of that point, two surfaces intersect one another. Triple points correspond to places where two of these framed curves cross one another transversely. Once we have established where the double point curves lie, we need only place surfaces connecting the appropriate double point curves together. The choices involved in placing the surfaces are all related by virtual replacements.

At this point we encounter an ambiguity in our choices which cannot be resolved by an isotopy in $S^3$ or a virtual replacement. The problem, which we do not encounter in the case of $1$-links, is that we have a choice in how to place the double point curves in $S^3$. The set of double point curves really defines a framed graph $\Gamma \subset S^3$, and there are many choices of framings and of isotopy classes of the graph. There is no reason to expect that we can pass between these choices using virtual Roseman moves. Therefore, if we wish to unambiguously use VBSDs to represent virtual $2$-links, we need to allow one more type of equivalence on top of the virtual Roseman moves. We must also allow ourselves to change the framing of a double point curve that is induced by the surfaces intersecting along the curve, and to pass double point curves through one another.

\begin{definition}
A \emph{graph change} to a VBSD is a change to the diagram which does not change the Gauss code of the double point curves, but which either passes a double point curve through another double point curve (or through itself), or changes the framing induced by the two intersecting surfaces along the double point curve.
\end{definition}

\begin{conjecture}
Virtual broken surface diagrams modulo the surface diagram Roseman moves, virtual Roseman moves, and graph changes represent virtual $2$-links.
\end{conjecture}

\begin{conjecture}
Graph changes cannot in general be obtained using Roseman moves and virtual Roseman moves.
\end{conjecture}

Note that if true, these conjectures would imply that our approach to virtual $2$-links gives a distinct theory from that given by Schneider in \cite{Schneider}, as the latter definition does not involve graph changes. Thus, our geometric generalization of virtual links to virtual $2$-links would be distinct from Schneider's combinatorial generalization. On the other hand if this conjecture is false, this would imply that the purely combinatorial definition given by Schneider is equivalent to our geometric definition. In either case, a deeper understanding of these graph changes would be beneficial to further research on representing virtual $2$-links with virtual broken surface diagrams.

%
\section{Fox-Milnor Movies}
\label{s-Fox-Milnor}
%

For surface links in $\mathbb{R}^{4}$, it is useful to consider descriptions in terms of a foliation of $\mathbb{R}^4$ into $\mathbb{R}^3 \times \mathbb{R}$. Such a description can be applied to a surface in a $4$-sphere by deleting one point from the sphere. Suppose that we consider $\mathbb{R}^4$ to have coordinates $(x, y, z, t)$, and suppose that the function $h(x,y,z,t)=t$ is a Morse function when restricted to a given surface link $L$. Then for generic values of $t$, the intersection of $L$ with the hyperplane of constant $t$ will be a classical link in three dimensions. At critical points of the Morse function, we will see the birth or death of a circle, or a crossing of two lines corresponding to a saddle point for the surface. A \emph{Fox-Milnor movie} of $L$ is the sequence of links and singular links ordered by their $t$-values and marking which arcs are involved in the births, deaths, and saddle point crossings. There will be a finite set of links and singular links, since the Morse function has finitely many critical levels. From such a movie, it is possible to reconstruct $L$. For such movies, Carter and Saito give a complete set of moves relating any two Fox-Milnor movies of isotopic links, \cite{CS}.

We may generalize this notion to elements of $V_2$ in a straightforward way by considering, for any pair $(F\times [0,1], L)\in V_2$, a Morse function $h_0$ on $F$, which defines a Morse function $h:F\times [0,1]\rightarrow \R$ by $h(x, t)=h_0 (x)$ for any $(x, t)\in F\times [0,1]$. If we perturb this Morse function so that its level sets have at most one critical point, then generically, the level sets of $h$ will be elements of $V_1$. Critical levels will correspond to either a change in $L$, during which we add or remove a trivial component to the element of $V_1$ or encounter a saddle point on the diagram $\pi(L)$, or will correspond to a change in the topology of the slice of $F$. Such objects, although their use as a generalization of virtual links in higher dimensions was not studied, were used in \cite{StableEq} to study notions of cobordisms and link homology for virtual links. These definitions admit straightforward generalizations to arbitrary dimensions.

\begin{definition}
Let $(F_i\times [0,1], L_i)$, $i=1, 2$, be elements of $V_n$. These elements are \emph{(virtually) link homologous} iff there exists an element $(F\times [0,1], L)\in V_{n+1}$ with a Morse function $h:M\rightarrow \R$ such that $(F_i\times [0,1], L_i)=h^{-1}(i)$. If in addition $L$ is diffeomorphic to $L_i\times [0,1]$, these elements are \emph{(virtually) link concordant}.
\end{definition}

\begin{theorem}
Let $(F_i\times [0,1], L_i)$, $i=1, 2$, be virtually equivalent elements of $V_n$. Then they are virtually link concordant.
\end{theorem}
\bpr The proof for $V_1$ may be found in \cite{StableEq}. In general, any virtual equivalence can be broken down into a series of surgeries, excisions, and gluings on the underlying manifold $M_i$ together with isotopy of $L_i$.\epr

It is also possible to express virtual broken surface diagrams in $\R^3$ by using Fox-Milnor movies of virtual $1$-link diagrams, together with a foliation of $\R^3$ by planes.

\subsection{Yoshikawa Diagrams}
\label{s-Yoshikawa}
%

In this approach to describing knotted surfaces in $4$-space, we consider Yoshikawa diagrams, sometimes called ch-diagrams or hyperbolic splittings, which are link diagrams where in addition to over and undercrossings, we can also have transverse intersections with an A-smoothing direction marked on them, such that the links obtained by performing all marked A (or B) smoothings results in an unlink. We will define a \emph{virtual Yoshikawa surface link} to be a Yoshikawa diagram (allowing virtual crossings) for which the A and B smoothed links are virtual unlinks. We may take these diagrams modulo Yoshikawa's moves and virtual replacement (which for a virtual link means to replace a segment of a curve with only virtual crossings by a different segment with the same endpoints and only virtual crossings) to define virtual Yoshikawa surface links.

A Yoshikawa diagram defines a (perturbed) Fox-Milnor movie by and, hence, an embedding of a surface into $\mathbb{R}^{4}$ by interpreting the A smoothing as the $t=0$ level set, the B smoothing as the $t=2$ level set, and the knotted graph itself as the $t=1$ level set, assuming that all saddle points occur at $t=1$. To make this into a true Fox-Milnor movie it is necessary to then perturb the saddle points slightly to ensure that they occur at different levels so that $t$ is in fact a Morse function. It is possible to express every knotted surface using a Yoshikawa diagram, \cite{CKS}.

Yoshikawa introduced a set of moves for Yoshikawa diagrams which were conjectured to be complete, meaning that two Yoshikawa diagrams are conjectured to represent the same surface iff they are related by some finite sequence of these moves. These moves consist of the usual Reidemeister moves together with those shown in Fig. \ref{Ymoves}.

\begin{figure}
		\centering
			\includegraphics[scale=0.3]{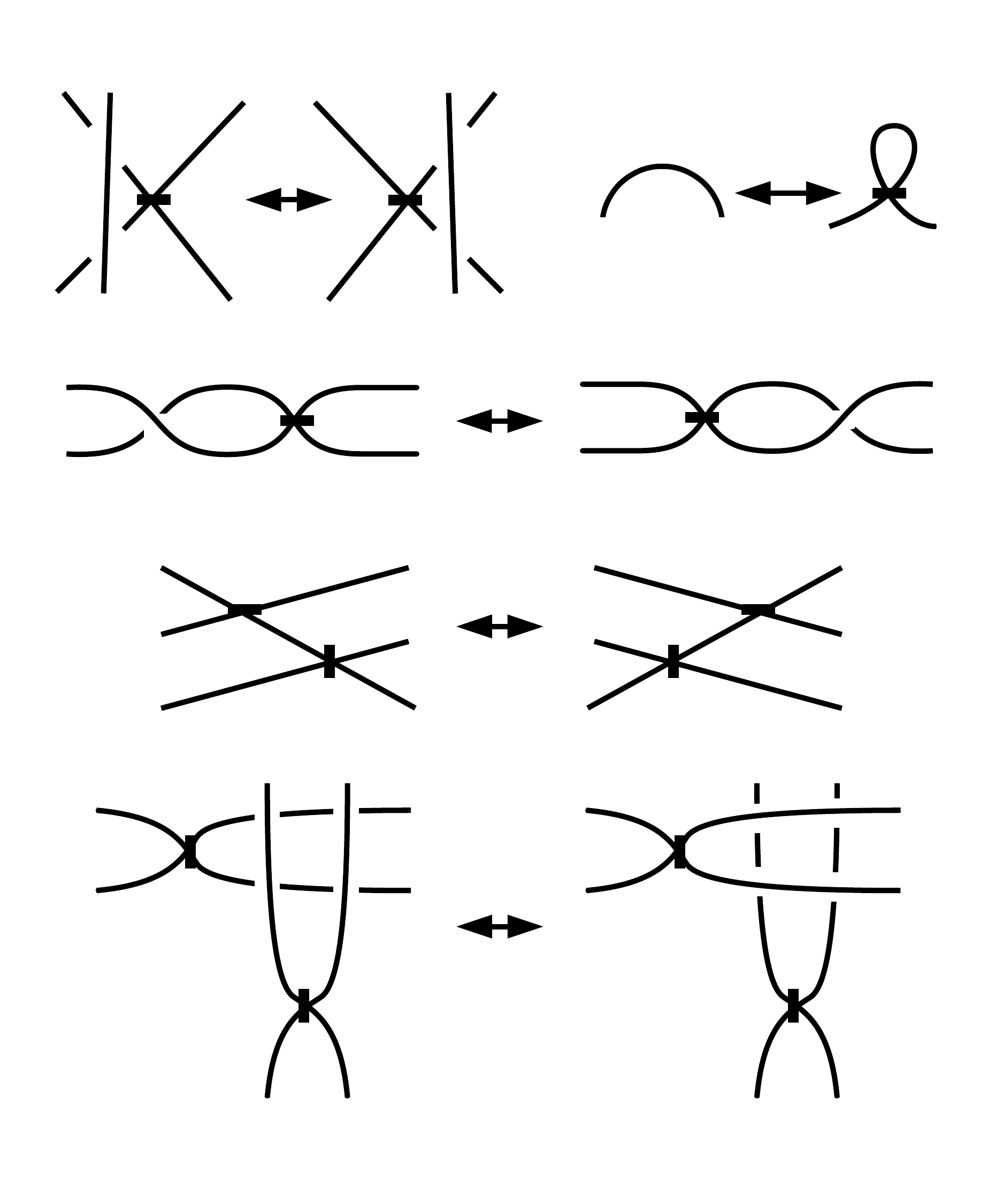}
		\caption{The additional moves permitted for a Yoshikawa diagram.}
		\label{Ymoves}
\end{figure}

By allowing virtual link diagrams to have marked crossings, and allowing all the virtual Reidemeister moves together with the moves shown in Fig. \ref{Ymoves}, we may use such diagrams to encode virtual Fox-Milnor movies for virtual surfaces, and thus to represent virtual broken surface diagrams in $\R^3$. These diagrams may be regarded as easier to utilize for some purposes, since they consist simply of singular virtual links with marked singularities. However, it is not clear whether these moves are complete, i.e. whether two virtual Yoshikawa diagrams which represent the same virtual $2$-link will necessarily be related by some sequence of these moves. Since these moves do not take into account the graph changes to the double point sets, it seems unlikely that they would suffice. Nor is it clear how graph changes could be incorporated into this formalism. Further study on such diagrams for virtual surfaces would be very helpful in understanding whether they could be a useful tool.

\begin{remark}
Since the combinatorial equivalence relation on virtual broken surface diagrams given by Schneider in \cite{Schneider} does not involve graph changes, it seems more likely that the Yoshikawa moves might provide a complete set of moves up to Schneider's stricter equivalence relation on VBSDs.
\end{remark}
%
\section{Realizability of Surface Gauss Codes}
%

Given a (classical) Gauss code, it is possible to determine whether or not this is the Gauss code of some knot diagram on the plane.

By a \emph{graph diagram} we mean a link diagram for a link whose components may be $4$-valent graphs instead of circles.

\begin{theorem}
Let $X$ be a surface Gauss code. Then $X$ is the surface Gauss code of some classical broken surface diagram in $\mathbb{R}^{3}$ iff $X$ admits a Morse function such that every level set has a Gauss code which is realizable as a graph diagram on the plane.
\end{theorem}
\bpr Suppose that $X$ is the Gauss code for some classical broken surface diagram in the $3$-hyperplane. Then a Fox-Milnor movie will define the required Morse function. On the other hand given such a Morse function $f$ on $X$, $f$ will define a Fox-Milnor movie each of whose frames is a classical link or graph. This Fox-Milnor movie thus defines a broken surface diagram in $\mathbb{R}^{3}$.\epr

In fact since every realizable broken surface diagram admits a self-indexing Morse function, it follows that a surface Gauss code is realizable iff it admits a self-indexing Morse function whose level sets for $t\neq 1$ are realizable links and which is a realizable knotted graph for $t=1$.

%

%

\section{Welded Links and Links in Higher Dimensions}
\label{s-welded}
%

In this section we will discuss how it may be possible to extend the definition of welded links to higher dimensions. However, we will not give a formal definition here. Instead, this discussion merely indicates some possible avenues for extending the definition.

Recall that, by the work of Satoh, \cite{SS}, welded $1$-links may be interpreted as ribbon embeddings of tori in $S^4$. We review his construction and main results here. Let $W$ be a welded link diagram. From this welded link we can construct a (classical) broken surface diagram by following the prescription illustrated in Fig. \ref{SatohTube}. The collection of tori in $S^4$ corresponding to this broken surface diagram will be denoted $Tube(W)$.

\begin{figure}
		\centering
			\includegraphics[scale=0.8]{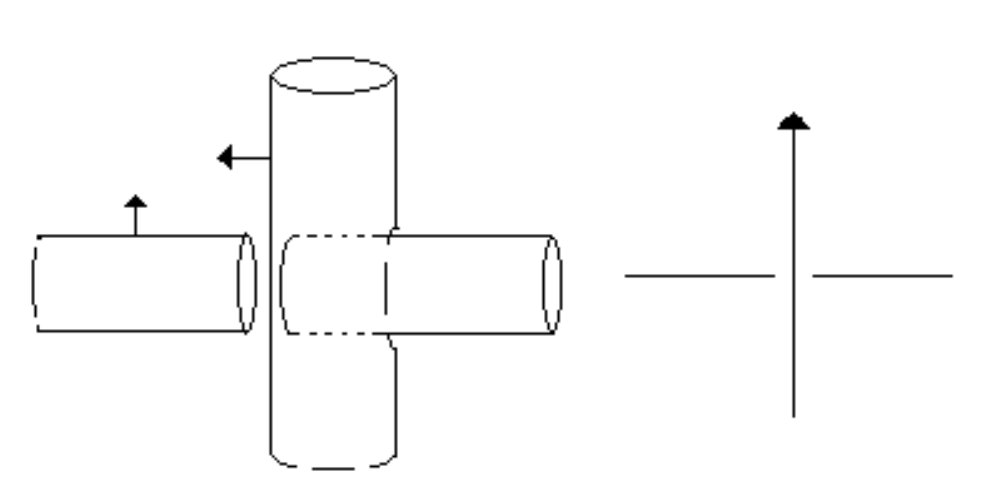}
		\caption{The definition of Satoh's $Tube$ map.}
		\label{SatohTube}
\end{figure}

The following two theorems are shown in \cite{SS}.

\begin{theorem}
If $W$ and $W'$ are welded link diagrams which are welded equivalent (i.e. related by welded Reidemeister moves), then $Tube(W)$ is isotopic to $Tube(W')$.
\end{theorem}

This theorem is proved by simply checking each welded Reidemeister move, and showing that there is a corresponding isotopy of the surface in $S^4$.

\begin{theorem}
For all ribbon tori $K$ in $S^4$, there is a welded knot $W$ such that $Tube(W)=K$.
\end{theorem}

One of the open questions in \cite{SS} is: do there exist welded knots $W, W'$ such that $W$ and $W'$ are not welded equivalent, but $Tube(W)$ is isotopic to $Tube(W')$? There is an infinite family of pairs of such examples given by the author in \cite{BW}. Let $W^\dagger$ denote the welded knot diagram obtained by reflecting the diagram of $W$ across some axis in the plane. Let $K^*$ denote the reflection of a surface in $S^4$. We will use $-$ to indicate reversing the orientation.

\begin{lemma}
For any welded knot $W$, $Tube(-W^\dagger)\cong -Tube(W)^*$.
\end{lemma}
This is proved in \cite{BW} by checking the resulting diagrams.

\begin{corollary}
Let $W$ be a welded knot which is not $(-)$ amphichiral, i.e. for which $W$ is not welded equivalent to $-W^\dagger$. Then $W, -W^{\dagger}$ provides an example of a pair of welded knots which are not welded equivalent, but whose corresponding ribbon tori under the $Tube$ map are isotopic.
\end{corollary}

Since there are many welded knots which are not $(-)$ amphichiral, this provides a large collection of examples. However, these welded knots are related by a fairly simple symmetry. Therefore, we may modify the question in \cite{SS} to the following form:

\begin{question}
Are there welded knots $W, W'$ such that $W$ and $W'$ are not welded equivalent \emph{and are not related by symmetries of reversal and mirror images}, but $Tube(W)$ is isotopic to $Tube(W')$?
\end{question}

For classical knots $K, K'$, it is known that $Tube(K)\cong Tube(K')$ iff $K\cong K'$ or $K\cong -K^\dagger$, \cite{Liv, BW}

Satoh's construction suffices to present not only ribbon tori using welded knots, but links whose components are ribbon tori by using welded knots.

\begin{remark}
Satoh constructs a link whose components are ribbon tori in $S^4$ for any welded link. However, the construction can be generalized to define a map $Tube_n$ on welded knots which assigns a ribbon presentation $(B, H)$ for an $n$-knot such that $|B|=|H|$ for any welded knot. Specifically, we let $Tube_n(W)=(B_W, H_W)$ be the ribbon presentation for $Tube(W)$, letting each arc stand for a base, and connecting these bases with handles in the obvious manner. The above results may easily be seen to apply to $Tube_n$ provided $n\geq 2$. $Tube_1$ is of course not well-defined.
\end{remark}

Satoh's work relating welded knots to ribbon knottings of $S^1\times S^{n-1}$ in $S^{n+2}$, $n\geq 2$, provides us with a geometric heuristic for defining higher-dimensional welded knots. As welded knots represent knotted tori in $\R^4$, it is reasonable to attempt to generalize this notion to higher dimensions by considering a Fox-Milnor movie of welded $1$-links. This will present a knotted $3$-manifold in $\R^5$. At singular levels, in which two arcs in the welded link are joined by a $1$-handle, we interpret the singularity as shown in Fig. \ref{Tube2}.

\begin{figure}[htbp]
\begin{centering}
\includegraphics[scale=.5]{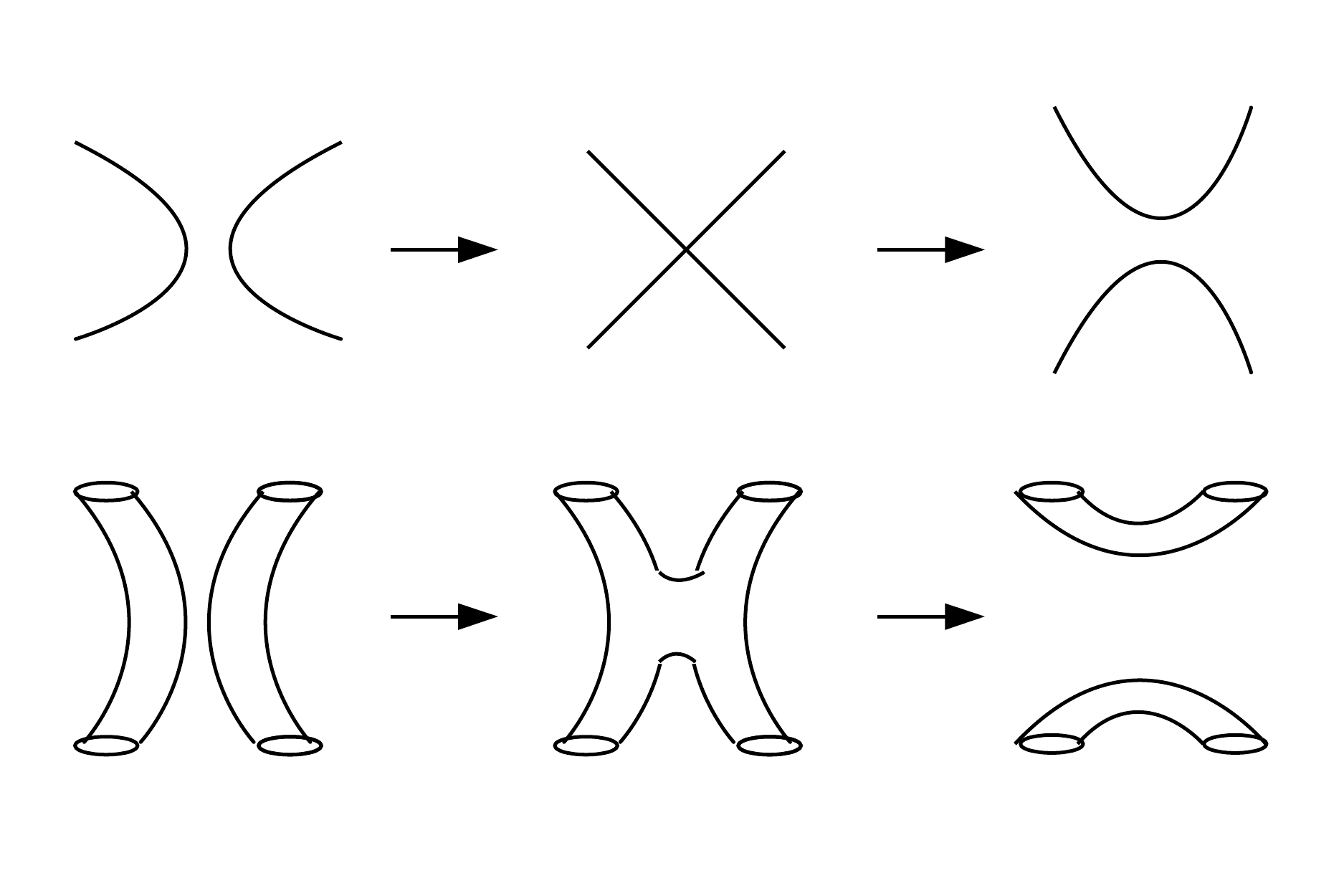}
\caption{The Fox-Milnor movie of a welded $1$-link $K$ shown in a neighborhood of a singular level, and the corresponding interpretation as a Fox-Milnor movie for a $3$-knot, the latter presented as a Fox-Milnor movie of broken surface diagrams.}
\label{Tube2}
\end{centering}
\end{figure}

A Fox-Milnor movie of welded knots can also be interpreted as a kind of broken surface diagram with classical and welded crossings. However, unlike the virtual broken surface diagrams, these welded broken surface diagrams will have triple points at which two classical double point curves meet a welded double point curve. Recall that for virtual broken surface diagrams, every triple point is either the meeting of three classical double point curves or one classical double point curve with two virtual double point curves. Therefore, not every Fox-Milnor movie of welded links will define a Fox-Milnor movie of virtual links. It is also not clear whether the generalized graph changes to a virtual broken surface diagram can be interpreted as isotopies of a $3$-knot in $\R^5$. As such, if we wish to extend the definition of welded knots to higher dimensions in such a way that welded $n$-links still represent $(n+1)$-links with certain geometric properties, it does not appear that welded links can be defined as a simple quotient of virtual links. The fact that for $1$-links, welded links are a quotient of virtual links would then appear as a special case.

It may be the case, then, that in higher dimensions, the geometric generalization of virtual knots which we have given will not be directly related to a generalization of welded knots that respects Satoh's geometric interpretation of welded $1$-links. More work is needed to determine how best to formally extend the notion of a welded knot to higher dimensions.

\section{Ribbon $n$-Knots with Isomorphic Quandles}
\label{s-Ribbon-Isom-Q}
%

Let us say that in a ribbon presentation $(B, H)$ a handle $h\in H$ is a \emph{trivial handle} if both ends are attached to the same base and $h$ can be isotoped so it does not meet any bases except at its ends.

\begin{definition}
Let $K, K'$ be $n$-knots in $S^{n+2}$, and let $T=S^1\times S^{n-1}$. Let $T_n$ indicate the connected sum of $n$ copies of $T$ that are trivially embedded, that is, which bound embedded copies of $S^1\times D^{n}$. We will say that $K$ and $K'$ are \emph{weakly equivalent} if $K\# T_n$ is isotopic to $K'\# T_m$.
\end{definition}

Note that in \cite{Liv}, weak equivalence is called ``stable equivalence,'' as gluing on $T$ may be thought of as a stabilization of the knot. Because this term is also used to refer stable equivalence of ribbon presentations, we are adopting the term ``weak equivalence'' instead to avoid any confusion.

\begin{definition}
Let $(B, H), (B', H')$ be ribbon presentations for $n$-links, $n\geq 2$. We will say that they are \emph{$k$-weakly stably equivalent} if they are stably equivalent after attaching $k$ trivial handles, and that they are \emph{weakly stably equivalent} if they are $k$-weakly stably equivalent for some finite $k$.
\end{definition}

Given a ribbon presentation $(B, H)$, the bases $b_i$ in $B$ are a collection of embedded $(n+1)$ disks. We will say that a path in $S^{n+2}-K$ \emph{passes through} a base $b_i\in B$ if it intersects the interior of $b_i$. Note that in general position, a path can pass through a base positively or negatively, depending on whether it passes through it in the same direction as the normal vector or in the opposite direction.

\begin{remark}
To specify a handle $h\in H$, it suffices to specify the starting base to which $\{ 0\} \times D^n$ is glued, the signed intersections of the handle in the order it passes through bases, and the terminal base to which $\{ 1\} \times D^n$ is glued. Any two handles meeting these conditions are smoothly isotopic and thus equivalent.
\end{remark}

\begin{theorem}
Let $K$ be a ribbon knot in $S^{n+2}$, $n\geq 2$, of arbitrary genus. Suppose that $K$ admits a base and handle presentation $(B, H)$ with $k$ bases, and $\pi_{1}(S^{n+2}-K)\cong \mathbb{Z}$. Then $(B, H)$ is $(k-1)$-weakly stably equivalent to an unknotted ribbon knot.
\end{theorem}
\bpr Let $h_{2},..., h_{k}$ be $(k-1)$ new trivial handles, with both ends connected to base $B_{1}$. Perform handle slides so that $h_{i}$ starts on $B_{1}$ and terminates on $B_{i}$, twisting them at the end so that they pass through bases equally many times positively as negatively. Let $\gamma_{i}$ be a curve which travels to $B_{1}$ without passing through any bases, then follows the handle $h_{i}$ to $B_{i}$ (traveling through the same bases in the same order) and returns to the basepoint for the fundamental group without passing through any additional bases. Since $\pi_{1}(S^{n+2}-K)\cong \mathbb{Z}$, and since by construction $\gamma$ is homologically trivial, $\gamma$ bounds a disk in $S^{n+2}-(K\cup h_{2}\cup ... \cup h_{k})$. This disk may be taken to be smooth and immersed by various approximation theorems, \cite[Theorem 10.16, Theorem 10.21]{Lee} and \cite[Theorem 2.5]{Adachi}. By general position, this disk meets itself at most in isolated points. Sliding $h_{i}$ along this disk we may, by handle passes only, make $h_{i}$ into a handle connecting $B_{1}$ to $B_{i}$ without passing through any bases. Now all handles in $H$ may be moved by handle passes so they only pass through base $B_{1}$, after which all other bases may be deleted after appropriate handle slides. But a ribbon presentation with a single base is always stably equivalent to the unknotted ribbon presentation. \epr

For a ribbon $n$-knot $K$ with $n\geq 2$ and ribbon presentation $(B, H)$, there is a natural presentation for the knot quandle $Q(S^{n+2}, K)$. This presentation is generated by the bases in $B$ and has one relation for each handle $h_{j}\in H$. The relation $r_{j}$ has the following form: let $b_s$ and $b_e$ be the starting and ending bases for $h_j$, and let $h_j$ follow the word $w=b_1 b_2... b_n$ (here $w$ is a word in the bases in $B$, where each base can appear positively or negatively, depending on whether the handle passes through it in a direction that agrees with the normal vector or opposite to it) as it moves from $b_s$ to $b_e$. Then the relation $r_j$ is given by the equation $b_e=b_s^{\overline{b_1}\overline{b_2}...\overline{b_n}}$. Geometrically, the generator corresponding to a base $b\in B$ is given by an equivalence class of paths containing a path that goes from the basepoint to the base $b$ without passing through any other bases.

\begin{theorem}
A ribbon $n$-knot, $n\geq 2$, with ribbon presentation $(B, H)$, has a knot quandle whose presentation is as described in the previous paragraph.
\end{theorem}
\bpr
This is proved by repeated application of the Van Kampen theorem. Alternately, it is not difficult to write down the double point curves in a projection of a ribbon knot where each base and handle is embedded under the projection. Then the double point curves correspond to handles passing through bases, and the relations may be worked out using Thm. \ref{KW}.
\epr

\begin{remark}
Let $Q(S^{n+2}, K)$ be the quandle of a ribbon knot $K$ with presentation $(B, H)$. To specify an element of $Q(S^{n+2}, K)$, it suffices to specify a path $\gamma$ by assuming that $\gamma$ starts at the basepoint for the quandle, then specifying the ordered signed intersections of the path with the bases, and finally specifying on which base the path terminates.
\end{remark}

This follows from general position arguments.

\begin{remark}
We will denote the generator of a quandle corresponding to a base $b$ by using the same symbol; whether we are referring to the base or the quandle element will be clear from context.
\end{remark}

Throughout the following, we will make extensive use of the following fact, which forms part of the geometric definition of the quandle of a knot complement. Let $K$ be a knot in $S^{n+2}$, and let $x$ be the basepoint for the quandle $Q(S^{n+2}, K)$. Let $N(K)$ be the closure of a tubular neighborhood of $K$. Recall that elements of the quandle are equivalence classes of paths $\gamma$ which start on $x$ and terminate on $\partial N(K)$. Two paths $\gamma, \gamma'$ are equivalent as quandle elements iff they are homotopic through paths starting on $x$ and terminating on $\partial N(K)$. Therefore, there exists a disk $D_{\gamma, \gamma'}$ whose boundary consists of the union of $\gamma \cup \gamma'$ with some path on $\partial N(K)$. By various approximation theorems, \cite[Theorem 10.16, Theorem 10.21]{Lee} and \cite[Theorem 2.5]{Adachi}, we may assume this disk is smooth and immersed, and therefore by general position we may assume that $D_{\gamma, \gamma'}$ is immersed and embedded except at a finite number of points, provided that $n\geq 2$. In fact for $n\geq 3$, general position arguments show that the disk can be taken to be embedded. Thus, the homotopy may be taken to be a homotopy through embedded paths, that is, a smooth isotopy of the path which traces out the disk $D_{\gamma, \gamma'}$.

We will repeatedly make use of the existence of such isotopies. All our applications are essentially uses of the following construction. Let $(B, H)$ be a ribbon presentation for $K$. Let $h, h'$ be handles such that $(B, H\cup \{h\})$ and $(B, H\cup \{h'\})$ are ribbon presentations. Recall that a handle is an embedding of $D^{1}\times D^{n}$ into $S^{n+2}$ such that $\{0,1\}\times D^{n}$ is embedded in elements of $B$. We will refer to $h|_{[0,1]\times\{0\}}$ as the \emph{core} of $h$. Observe that once the core of $h$ is specified, the handle itself is specified up to isotopy. Now suppose that $h, h'$ have the property that $h|_{[0,1/2]\times\{0\}}=h'|_{[0,1/2]\times\{0\}}$. Let $\gamma=h|_{[1/2,1]\times\{0\}}, \gamma'=h'|_{[1/2,1]\times\{0\}}$. Now $\gamma, \gamma'$ determine quandle elements for $Q(S^{n+2}, K)$.

\begin{lemma}
Let $K$, $(B, H)$, $h, h', \gamma, \gamma'$ be as described in the previous paragraph, and suppose that $\gamma\cong\gamma'$ as elements of $Q(S^{n+2}, K)$. Then there is a stable equivalence between $(B, H\cup \{h\})$ and $(B, H\cup \{h'\})$ which involves only performing handle slides isotoping $h$ to agree with $h'$.\label{qhs}
\end{lemma}
\bpr Let $D_{\gamma, \gamma'}$ be the homotopy disk giving the equivalence of $\gamma, \gamma'$ as elements of $Q(S^{n+2}, K)$. This is a disk in the complement of $K$. This disk may be perturbed so that its interior does not meets $h$ or $h'$, by general position arguments. Now $D_{\gamma, \gamma'}$ intersects itself only in isolated points. Therefore, we may slide $h|_{[1/2, 1]\times\{0\}}$ along this disk until it agrees with $h'|_{[1/2, 1]\times\{0\}}$. During this process we will be sliding the end of the core of $h$ around on $\partial N(K)$ so we may perform some handle slides. \epr

\begin{remark}
This argument cannot be extended to an argument for an analogous result for handles $h$ and $h'$ that are already part of the ribbon presentation for $K$. This is because in that case, the disk $D_{\gamma, \gamma'}$ might pass through the interior of the handle $h$ or $h'$ themselves. This cannot happen in the above case because $\gamma, \gamma'$ are homotopic using a homotopy in the complement of $K$, and, therefore, the homotopy cannot move the path through the interior of $h$ or $h'$ since they are not handles for the presentation of $K$.
\end{remark}

\begin{lemma}
Let $K$ be a ribbon $n$-knot, $n\geq 2$, with ribbon presentation $(B, H)$. Suppose that the quandle $Q(S^{n+2}, K)$ is generated by $B-\{ b\}$. Then, by adding a trivial handle to $(B, H)$, we can merge $b$ and some $b'\in B-\{ b\})$ into a single base through stable equivalence.\label{reduce}
\end{lemma}
\bpr Let $h$ be the new trivial handle, with both ends on $b$. We may isotope $h$ so that the midpoint of the core of $h$ passes through the basepoint $x$ for the quandle while it remains a trivial handle. Now the path that goes from the basepoint to $b$ is equivalent in the quandle to a path $\gamma$ from the basepoint to some other base $b'$ which does not pass through the base $b$. There is a quandle homotopy disk connecting these two paths which does not intersect $h$ (since all the relations in the quandle are geometrically given by the handles in $H$). We may slide the portion of $h$ between its midpoint and terminal point along this disk so that it joins $b$ and $b'$ and follows the path $\gamma$, using the method described in Lemma \ref{qhs}. Since $\gamma$ does not pass through $b$, we can now perform handle passes through $h$ so that any handles passing through $b$ are moved to instead pass through $b'$. Then the two bases can easily be joined. \epr

\begin{corollary}
Every ribbon $n$-knot $K$, $n\geq 2$, with ribbon presentation $(B, H)$, is $(|B|-1)$-weakly stably equivalent to a ribbon $n$-knot with presentation $(B', H')$, where $B'\subseteq B, H\subseteq H'$, with the property that for all $b\in B'$, the subquandle generated by $B'-\{ b\}$ is a proper subquandle of $Q(S^{n+2}, K)$.\label{redq}
\end{corollary}

\begin{theorem}
Let $K, K'$ be ribbon $n$-knots, $n\geq 2$, with ribbon presentations $(B, H), (B, H')$, respectively, of the same genus, such that the identity map on $B$ induces a quandle isomorphism between $Q(S^{n+2}, K)$ and $Q(S^{n+2}, K')$. Then they are $|H|$-weakly stably equivalent. \label{isoeq}
\end{theorem}
\bpr Let $h'\in H'$ connect $b_0, b_1$. Attach a trivial handle $h$ to $H$ with both ends on $b_0$. Let $h$ be isotoped so that the midpoint of its core passes through the basepoint $x$ for the quandle. The path from $x$ to $b_0$ following the core of $h$ and the path from $x$ to $b_1$ following $h'$ are homotopic in the quandle $Q(S^{n+2}, K)$; thus we may slide $h$'s latter half along the homotopy disk, as described in Lemma \ref{qhs}. After this is done, $h$ will pass through exactly the same bases as $h'$. Repeat this for each handle in $H'$. Then perform the same operation using the trivial handles we have added to $H'$, modifying them so that all the handles originally in $H$ are now in $H'$ as well. We now have all the handles in $H$ in $H'$ as well, and all the handles in $H'$ have been added to $H$. Thus we have modified the two ribbon presentations so they have identical sets of bases and handles. \epr

\begin{theorem}
Let $K, K'$ be ribbon $n$-knots, $n\geq 2$, with ribbon presentations of the same genus $(B, H), (B', H')$, respectively. Let $f:B\rightarrow Q(S^{n+2}, K')$ be a function which extends to a quandle isomorphism $Q(S^{n+2}, K)\rightarrow Q(S^{n+2}, K')$. Then $K, K'$ are $(2|B'|+|H|)$-weakly stably equivalent with bases given by $B$ and $f(B)$ respectively. \label{injeq}
\end{theorem}
\bpr For each $b_i\in B$, we obtain a quandle element $b'_i=f(b_i)\in Q(S^{n+2}, K')$. The $b_i'$ are paths from the basepoint to some base $B'_i \in B'$. Create a trivial base $B''_i$ connected by a handle to $B'_i$ by a handle that does not pass through any bases. Then we may slide $b_i'$ so that it terminates on $B''_i$. Clearly $b_i'$ does not pass through $B''_i$. Therefore, we can pull $B''_i$ back along $b_i'$, contracting the quandle element, until we have so isotoped $B''_i$ such that $b_i'$ is now a path from the basepoint to $B''_i$ which does not pass through any bases.

Now the collection of $B''_i$s generates the quandle $Q(S^{n+2}, K')$. By Lemma \ref{reduce}, $(B', H')$ is $|B'|$-weakly stably equivalent to a knot whose bases consist of exactly the $B''_i$s. Add an equal number of handles to $H$. Then by Theorem \ref{isoeq} the result follows. \epr

\begin{theorem}
Let $K, K'$ be ribbon $n$-knots (of the same genus) with ribbon presentations $(B, H), (B', H')$, respectively. Suppose the knot quandles $Q(S^{n+2}, K)$ and $Q(S^{n+2}, K')$ are isomorphic. Then $K$ and $K'$ are $(2|B|+|H|)$-weakly stably equivalent.
\end{theorem}
\bpr Let $f:Q(S^{n+2}, K)\rightarrow Q(S^{n+2}, K')$ be the isomorphism. Then $f$ defines a map from $B$ to $Q(S^{n+2}, K')$ defined by sending $b\mapsto f(b)$ (where we send the base to the quandle element which its standard generator is sent to by $f$. Now apply Theorem \ref{injeq}. \epr

\begin{corollary}
Let $K, K'$ be as in the previous theorem, but with possibly different genus. Then they are weakly stably equivalent.
\end{corollary}

\begin{corollary}
Let $(B, H), (B', H')$ be two ribbon presentations for a single $n$-knot $K$, $n\geq 2$. Then they are $(2|B|+||H|)$-weakly stably equivalent.
\end{corollary}

This is related to the question posed in \cite{Naka}, where it is asked whether two ribbon presentations for the same $n$-knot with $n\geq 2$ are necessarily stably equivalent. In \cite{NakaNaka} an infinite family of examples are constructed showing that two ribbon presentations for the same $1$-knot do not need to be stably equivalent. In fact examples are constructed of $1$-knots with $n$ stably inequivalent ribbon presentations for any $n\in\mathbb{N}$.

%

\section{Conclusion and Further Problems}
\label{s-conc}
%

We have explored a method for generalizing virtual links and tangles to higher dimensions, using the geometric definition for virtual links given in \cite{StableEq}. From this we have established that such higher-dimensional virtual knots have many invariants which are extensions of invariants for classical links. However, it is not true in higher dimensions that classical links embed into virtual links: we have given examples of virtually knotted tori which are classically distinct but virtually equivalent.

There are a number of questions which can be asked about these virtual $n$-links and tangles. We list here a few which may be of potential interest.

\begin{question}
Generalize Langford's and Baez's, \cite{Lang, BL}, categorical description of $2$-tangles to virtual $n$-tangles.
\end{question}

Langford and Baez, \cite{Lang, BL}, have given a description of classical $2$-tangles as a certain kind of $2$-category. It is reasonable to expect that a similar categorical description would apply to virtual $2$-tangles. In addition to being an interesting problem in itself, we may expect that such a categorical description might be useful in studying Khovanov homology of virtual knots. More generally, it should be possible to describe $n$-links and virtual or welded $n$-links as some type of $n$-category.

For $1$-links, the study of virtual links through tangles can be used to define certain quantum invariants of virtual links. These may relate to the case of higher $n$ as well.

\begin{question}
Find an example of a pair of knotted $n$-spheres in $D^{n+1}\times [0,1]$ which are not isotopic but which are virtually equivalent, or prove that no such examples exist.
\end{question}

The examples which we have given all involved knotted tori in $S^4$. We conjecture, however, that there should be examples involving knotted $2$-spheres as well. As mentioned in our discussion of the examples of tori, the knotted sphere examples given by Gordon in \cite{Gor} may provide a good starting place for investigating this conjecture. However, since many invariants of classical knots extend to invariants of virtual knots in all dimensions, finding such examples is nontrivial.

\begin{question}
Find examples of non-trivial virtual $n$-knots for which all homotopy groups $\pi_i(h(F)-K)$, $i=1... n$, are isomorphic to those of the $h(D^{n+1})-U_n$ (where $U_n$ denotes the trivial knot).
\end{question}

\begin{question}
Extend the definition of biquandle invariants to links of all dimensions.
\end{question}

Given the usefulness of biquandles for virtual $1$- and $2$-links (see e.g. Corollary \ref{2Kish}), it is reasonable to think that such an invariant may have similar applications for virtual $n$-links in general. We conjecture that these invariants might be able to be extended inductively to higher dimensions by using a technique like Fox-Milnor movies. Alternately, it may be possible to check the Roseman moves directly, although since the number of Roseman moves increases rapidly with $n$, it may not be possible to use this approach to give a proof for general $n$.

\begin{question}
Give a formal definition of welded links that applies in any dimension.
\end{question}

\end{document}